\documentclass{amsart}

\usepackage{amsfonts}
\usepackage{amscd,amsmath,amsthm,amssymb}
\usepackage{bm}
\usepackage{verbatim}
\usepackage{lscape}

\newcommand{\C}{\mathbb C}
\newcommand{\R}{\mathbb R}

\newcommand{\Z}{\mathbb Z}

\def\a'{\`a}
\def\e'{\`e}
\def\o'{\`o}
\def\u'{\`u}

\usepackage{geometry}
\begin{document}

 \title{$Sp(n)$-orbits in the Grassmannians of complex and $\Sigma$-complex subspaces of an Hermitian quaternionic vector space.}

\author{Massimo Vaccaro}

\address{Dipartimento dell'Ingegneria di Informazione e Matematica Applicata, Universit\a' di Salerno, 84084 - Fisciano (SA) , Italy}
\email{massimo\_vaccaro@libero.it}
\thanks{Work done under the programs of GNSAGA-INDAM of C.N.R. and PRIN07 "Riemannian metrics and differentiable structures" of MIUR (italy)}

\maketitle

\markboth{ \rm{MASSIMO VACCARO}} {\rm{$Sp(n)$-orbits in the Grassmannians  of complex and $\Sigma$-complex subspaces of an Hermitian quaternionic vector space}}


\newtheorem{teor}{Theorem}[section]
\newtheorem{coro}[teor]{Corollary}
\newtheorem{lemma}[teor]{Lemma}
\newtheorem{prop}[teor]{Proposition}
\newtheorem{samp}[teor]{Example} 
\newtheorem{defi}[teor]{Definition}
\newtheorem{exer}[teor]{Exercise}
\newtheorem{remk}[teor]{Remark}
\newtheorem{conj}[teor]{Conjecture}
\newtheorem{claim}[teor]{Claim}

\newtheorem{intteo}{Theorem}
\newtheorem{intcoro}{Corollary}

\textbf{Abstract.}
We determine the invariants characterizing the $Sp(n)$-orbits in the real Grassmannian $Gr^\R(2k,4n)$ of  the $2k$-dimensional  complex and $\Sigma$-complex subspaces of a $4n$-dimensional  Hermitian quaternionic vector space. A $\Sigma$-complex subspace is the orthogonal sum of complex subspaces by different, up to sign, compatible complex structure. The result is obtained by considering two main features of such subspaces. The first is that any such subspace admits a decomposition  into an Hermitian orthogonal sum of 4-dimensional complex addends plus a 2-dimensional totally complex subspace if $k$ is odd,  meaning that the quaternionification of the addends are orthogonal in pairs. The second is that any 4-dimensional complex addend $U$  is an isoclinic subspace  i.e.   the principal angles of the pair $(U,AU)$ are all the same for any compatible complex structure $A$. Using these properties we determine the full set of the invariants characterizing the $Sp(n)$-orbit of  such subspaces in $Gr^\R(2k,4n)$.

\vskip .3cm


\textbf{Summary.}
In this paper we study the $Sp(n)$-orbits in the real Grassmannians  $Gr^\R(2k,4n)$ of  some special $2k$-dimensional subspaces of a $4n$-dimensional real vector space $V$. We endow $V$ with  an Hermitian quaternionic structure ($\mathcal{Q},<,>)$, an  Hermitian product $( \,  \cdot \,)$  and denote by $S(\mathcal{Q})$ the 2-sphere of complex structures $J \in \mathcal{Q}$. The  invariants characterizing the $Sp(n)$-orbit of a generic subspace $U \subset V$ appear in \cite{Vacpreprint} where we find some equivalent statements. The first one  we report here  is given in the Theorem (\ref{transitivity of Sp(n) on subspaces with same Hermitian matrices}) where it is stated that   a pair of subspaces $U$ and $W$ of real dimension $m$ in the $\mathbb{H}$-module $V^{4n}$ belong to the same $Sp(n)$-orbit  iff there exist
  bases $\mathcal{ B}_U=(X_1, \ldots, X_m)$ and  $\mathcal {B}_W=(Y_1, \ldots, Y_m)$ of $U$ and $W$ respectively w.r.t. which for the Hermitian products one has  $(X_i \cdot X_j)= (Y_i \cdot Y_j), \; i=1, \ldots,m$ for one and hence any (hypercomplex) admissible  basis of $\mathcal{Q}$.

Let $\bm{\theta}^A(U)$ be the vector of the principal angles between the  pair $(U,AU), \; A \in S(\mathcal{Q})$ in non decreasing order. A consequence of the Theorem (\ref{transitivity of Sp(n) on subspaces with same Hermitian matrices}) is that a necessary condition for $U$ and $W$ to share the same $Sp(n)$-orbit is that, for one and hence any admissible basis $(I,J,K)$, one has   $\bm{\theta}^I(U)=\bm{\theta}^I(W),\; \bm{\theta}^J(U)=\bm{\theta}^J(W),\;, \bm{\theta}^K(U)=\bm{\theta}^K(W)$.

The determination of the  principal angles between a pair of subspaces $S,T$ is a well know problem solved by the singular value decomposition of the orthogonal projector  of $S$ onto $T$. Here, for a chosen $U \subset V$ we consider  the pairs $(U,AU), A \in S(\mathcal{Q})$  and denote by  $Pr^{AU}$ the orthogonal projector of $U$ onto $AU$. In this case the singular values of  $Pr^{AU}$ are always degenerate which implies that they have non-unique singular vectors. In terms of principal vectors of the pair $(U,AU)$ we can equivalently say that the  principal vectors are never uniquely defined.

Another way to obtain the principal angles and the associated principal vectors between the pair of subspaces $(U,AU)$,
for any $A \in S(\mathcal{Q})$,  is through the  standard decompositions of the restriction to $U$ of the $A$-K\"{a}hler skew-symmetric form  $\omega^A: (X,Y) \mapsto <X,AY>,  \; X,Y \in U$.
Calling \textit{standard basis} any orthonormal basis w.r.t. which $\omega^A|_U$ assumes standard form, with the non-negative entries ordered in non increasing order, and denoting by $\mathcal{B}^A(U)$ the set of all such bases one has that any $B \in \mathcal{B}^A(U)$ consists of  principal vectors of the pair $(U,AU)$.
As stated beforehand, for any $A \in S(\mathcal{Q})$  the standard bases of $\omega^A$ are never unique not even if $U$ is  2-dimensional.
This is the main problem in determining the $Sp(n)$-orbits in $Gr^\R(k,4n)$ as it is evident from
an equivalent conditions to the one stated in Theorem (\ref{transitivity of Sp(n) on subspaces with same Hermitian matrices}) appearing in  \cite{Vacpreprint} and  here reported in Theorem  (\ref{main_theorem 1_ existance of canonical bases with same mutual position}). According to it, a pair of subspaces $U,W$ are in the same $Sp(n)$-orbit iff, for one and hence any  admissible  basis $(I,J,K)$ of $\mathcal{Q}$ the following 2 conditions are satisfied:
\begin{itemize}
\item  $\bm{\theta}^I(U)=\bm{\theta}^I(W),\; \bm{\theta}^J(U)=\bm{\theta}^J(W),\; \bm{\theta}^K(U)=\bm{\theta}^K(W)$;
    \item there exist three orthonormal bases $(\{ X_i \} , \{Y_i \} ,\{ Z_i \}) \in (\mathcal{B}^I(U) \times \mathcal{B}^J(U) \times \mathcal{B}^K(U))$   and  $(\{X'_i \},\{Y'_i \},\{ Z'_i \}) \in (\mathcal{B}^I(W) \times \mathcal{B}^J(W) \times \mathcal{B}^K(W))$ whose relative position is the same or equivalently  \[A= A', \qquad B= B'\]
         where $A=(<X_i, Y_j>), \qquad A'=(<X_i', Y_j'>), \qquad B=(<X_i, Z_j>),\qquad  B'=(<X_i', Z_j'>)$.
\end{itemize}

The problem to determine the $Sp(n)$-orbits turns therefore into the one of determining the existence of such triple of  bases. In  \cite{Vacpreprint} we set up a procedure  to determine a triple of  \textit{canonical bases} w.r.t. which one  computes the matrices $A$ and $B$. Namely, fixed an admissible basis $(I,J,K)$, a triple of canonical bases of a subspace $U$ is constituted by
a triple of standard bases of $\omega^I,\omega^J,\omega^K$
that either  are  uniquely determined by the procedure aforementioned or, if this is not the case, nevertheless the associated matrices $A,B$ are unique.

We call the matrices  obtained thereof  \textit{canonical matrices} and denote them by  $C_{IJ}$ and $C_{IK}$. In \cite{Vacpreprint}, Chosen an admissible basis  $(I,J,K)$  we associate then to any subspace $U \subset V$ the  following invariant
\[Inv(U)=\{\bm{\theta}^I(U), \, \bm{\theta}^J(U), \,   \bm{\theta}^K(U),  C_{IJ}, \; C_{IK}  \},\] 
 and in the Theorem (\ref{main_theorem 1_ with respect canonical bases})  we affirm that 
the subspaces  $U$ and $W$ of $V$ are in the same $Sp(n)$-orbit  iff $Inv(U)=Inv(W)$. 
If $Inv(U)=Inv(W)$ w.r.t. the admissible basis $(I,J,K)$ then $Inv(U)=Inv(W)$ for any admissible basis.


After recalling  the definition of some  special subspaces of $V$, in Proposition (\ref{Decomposition of a generic subspace}), we show that a generic subspace $U$ of $(V,\mathcal{Q},< , >)$ admits a decomposition
\[
\begin{array}{l}
U^m= U_{Q} \stackrel{\perp} \oplus U^\Sigma \stackrel{\perp} \oplus U_R    \hskip 2.0cm \text{ with} \qquad \qquad
U^\Sigma=(U_1,I_1) \stackrel{\perp} \oplus \ldots \stackrel{\perp} \oplus (U_p,I_p)
\end{array}
\]
into an orthogonal sum of the maximal quaternionic subspace $U_Q$, a $\Sigma$-complex  subspace $U^\Sigma$, defined as the orthogonal sum of maximal $I_i$-complex subspaces $(U_i,I_i)$ with $I_i \in S(\mathcal{Q})$ and a totally real subspace $U_R$. In Proposition (\ref{orthogonal complex subspaces by different structures}) we prove that the complex addends $U_Q$ and $(U_i,I_i)$  are Hermitian orthogonal i.e. their quaternionifications $U_Q$ and
 $(U_i)^\mathbb{H}$  are orthogonal in pairs. In  general this  is not  true for the orthogonal totally real addend $U_R$.\\

  The subspaces we consider in this paper  are the $I$-complex subspaces $(U,I)$ with $I \in S(\mathcal{Q})$ and the $\Sigma$-complex subspaces.  It will turn out that  the analysis of the  4-dimensional complex case is fundamental to determine the $Sp(n)$-orbit of such subspaces.
   Any 4-dimensional $I$-complex subspace $(U,I)$   is characterized by the fact that, for any $A \in S(\mathcal{Q})$ the pair $(U,AU)$ is isoclinic.  In \cite{Vac_isoclinic} we denoted by $\mathcal{IC}^4$ the set of all 4-dimensional subspaces sharing this property and we called them  \textit{isoclinic subspaces}.
 We then recall the main results which concern a subspace $U \in  \mathcal{IC}^4$ referring to  \cite{Vac_isoclinic} for proofs and a wider treatment.



    In  \cite{Vac_isoclinic} we  determined the triple of canonical bases  $\{X_i\},\{Y_i\},\{Z_i\}, \; i=1, \ldots,4$  of $U$.
    Given an  admissible basis $(I,J,K)$ and chosen  $X_1 \in U$, we considered the standard 2-planes $U^I=L(X_1,X_2), \; U^J=L(X_1,Y_2), \; U^K=L(X_1,Z_2)$  centered on  a common unitary vector $X_1$ of the skew-symmetric forms $\omega^I,\omega^J,\omega^K$ respectively where
    \[X_2=\frac{I^{-1} Pr_U^{IU} X_1}{\cos \theta^I}, \quad Y_2=\frac{J^{-1} Pr_U^{JU}X_1}{\cos \theta^J}, \quad Z_2=\frac{K^{-1} Pr_U^{KU} X_1}{\cos \theta^K}.\]



    Denoted by $\xi =<X_2,Y_2>, \; \chi=<X_2,Z_2>, \; \eta=<Y_2,Z_2>$, in Corollary (\ref{invariance of <X_2,Y_2>})  we proved that the triple $(\xi,\chi,\eta)$ is an invariant of $U$.
  We introduced   $\Gamma$  as a function of $(\xi,\chi,\eta)$ and $\Delta= \pm \sqrt{1-\Gamma^2}$.  After proving that the pair $(\Gamma, \Delta)$ itself is an invariant of $U$,  in  the Proposition (\ref{canonical matrices of any  subspace of ${IC}^(4)$}) we affirm that the invariants $(\xi,\chi,\eta, \Delta)$ determine the canonical matrices $C_{IJ}$ and $C_{IK}$ which are given in (\ref{canonical matrices $C_{IJ}$ $C_{IK}$ and of 4-dimensional isoclinic subspace w.r.t. the associated chains}) w.r.t. the canonical bases.
  Therefore, according to the statement of the Theorem (\ref{main_theorem 1_ with respect canonical bases}), in the Theorem 
   (\ref{Sp(n) orbit of a 4 dimensional isoclinic subpace}) we state that the invariants $(\xi,\chi,\eta, \Delta)$ together with the angles $(\theta^I, \theta^J, \theta^K)$ determine the $Sp(n)$-orbit of any $U \in \mathcal{IC}^4$ in $G^\R(4,4n)$.



  The set of 4-dimensional complex subspaces is a subset of   $\mathcal{IC}^4$.  Let then  $(U,I)$ be a 4-dimensional $I$-complex subspace with $I \in S(\mathcal{Q})$.
  Fixed an adapted basis $(I,J,K)$,  we associate to $(U,I)$ the $I^\perp$- K\"{a}hler angle $\theta^{I^\perp}$ which is one of the four identical principal angles of the pair $(U,KU)$ observing that $KU$ is the same for any $K \in I^\perp \cap S(\mathcal{Q})$.
In this case  the angles of isoclinicity  $(\theta^I, \theta^J,\theta^K) =(0, \theta^{I^\perp},\theta^{I^\perp})$ and we denote such subspace by the triple $(U,I,\theta^{I^\perp})$. Furthermore, considered a triple of canonical bases $\{X_i\},\{Y_i\},\{Z_i\}, i=1, \ldots,4$  of the skew-symmetric forms $\omega^I,\omega^J,\omega^K$ centered on $X_1$  one has $\xi=\chi=\eta=0$ and the matrices $C_{IJ}$ and $C_{IK}$  of all 4-dimensional  complex subspaces are given in (\ref{canonical matrices of a 4 dimensional complex subspace}).
Then, according to the Theorem (\ref{main_theorem 1_ with respect canonical bases}), the pair $(I,\theta^{I^\perp})$  determines the $Sp(n)$-orbit of $U$ i.e. all and only the 4-dimensional $I$-complex subspaces with same $I^\perp$- K\"{a}hler angle constitute one $Sp(n)$-orbit in $G^\R(4,4n)$ as stated in Theorem (\ref{orbit of a complex 4-plane under Sp(n)}).



In Theorem (\ref{canonical decomposition of a complex subspace}) we affirm that a $2m$-dimensional $I$-complex subspace admits an Hermitian orthogonal   decomposition into 4-dimensional $I$-complex subspaces. Although such decomposition is not unique  we can associate to $U$ the canonically defined   \textbf{$I^\perp$-K\"{a}hler multipleangle} $\boldsymbol{\theta^{I^\perp}}=(\theta_1^{I^\perp}, \ldots, \theta_{[m/2]}^{I^\perp})$ ($\boldsymbol{\theta^{I^\perp}}=(\theta_1^{I^\perp}, \ldots, \theta_{{m/2}}^{I^\perp}, \pi/2)$ if $m$ is odd)  of the $I$-complex $2m$-dimensional subspace $(U,I)$  being $\boldsymbol{\theta}^{I^\perp}$  the set of the $I^{\perp}$-K\"{a}hler angle of the Hermitian orthogonal 4-dimensional $I$-complex addends (plus  the $K$-K\"{a}hler angle of an Hermitian orthogonal totally $I$-complex plane if $m$ is odd with $K \in I^\perp$).
Denoted by $Gr_{(I,\boldsymbol{\theta}^{I^\perp})}^\R(2m,4n)$ the set of $2m$-dimensional  pure $I$-complex subspaces in $(V^{4n}, <,>, \mathcal{Q})$ with $I^\perp$-K\"{a}hler multipleangle $\boldsymbol{\theta^{I^\perp}}$, in  Theorem (\ref{orbit of a complex 2m-plane under Sp(n)}) we state that the group $Sp(n)$ acts transitively on $Gr_{(I,\boldsymbol{\theta}^{I^\perp})}^\R(2m,4n)$ i.e.
the pair $(I, \boldsymbol{\theta}^{I^\perp})$ composed by the complex structure $I \in \mathcal{Q}$ and the $I^\perp$-K\"{a}hler multipleangle  $\boldsymbol{\theta}^{I^\perp}$ of the $I$-complex subspace $U$ determines completely its $Sp(n)$-orbit in the Grassmannian $Gr^\R(2m,4n)$.

In particular, all totally $I$-complex subspaces of same dimension  form one orbit in $G^\R(2m,4n)$.


We  then consider a $\Sigma$-complex subspace $U$.  
 From Proposition (\ref{unicity of the decomposition of a subspace with trivial real addend}) the  decomposition of $U$ into an orthogonal sum of maximal pure complex subspaces by different (up to sign) structures  is unique. Moreover, from Proposition (\ref{orthogonal complex subspaces by different structures}), the $2m_i$-dimensional $I_i$-complex subspaces are Hermitian orthogonal and to determine the $Sp(n)$-orbit we can deal separately with each  $I_i$-complex  addend. In Proposition (\ref{orbit af a generic subspace with trivial real part under Sp(n)}) we state that the pair $(\mathcal{I}, \Theta)$, where $\mathcal{I}= \{ I_i\}$ and $\Theta= \{\boldsymbol{\theta_i^{I_i^\perp}}\}$ is the vector whose elements are the $I_i^\perp$-K\"{a}hler multipleangle of the $I_i$-complex addend, completely determines the orbit in the real Grassmannian.\\

\section{Decomposition of a generic subspace of an Hermitian quaternionic vector space}
\subsection{The Hermitian quaternionic structure}

Let $V$ be a real vector space of dimension $4n$.
\begin{defi} $ $
\begin{enumerate}
\item A triple $\mathcal{H}=\{J_1,J_2, J_3 \}$ of anticommuting complex structures  on $V$ with $J_1J_2= J_3$ is called a \textbf{hypercomplex structure} on $V$.
\item The 3-dimensional subalgebra
\[ \mathcal{Q}= span_\R(\mathcal{H})= \R J_1 + \R J_2 + \R J_3 \approx \mathfrak{sp}_1 \]
of the Lie algebra $End(V)$ is called a \textbf{quaternionic structure} on $V$.
\end{enumerate}
\end{defi}
Note that two hypercomplex structures $\mathcal{H}= \{J_1,J_2,J_3\}$ and $\mathcal{H'}= \{J'_1,J'_2,J'_3\}$ generate the same quaternionic structure $\mathcal{Q}$ iff they are related by a rotation, i.e.
\[
J'_\alpha= \sum_\beta A^\beta_\alpha J_\beta,\qquad (\alpha= 1,2,3)
\]
with $(A_\alpha^\beta) \in SO(3)$. A  hypercomplex structure  generating  $\mathcal{Q}$ is called an \textbf{admissible} \textbf{basis} of $\mathcal{Q}$.
\noindent We denote by $S(\mathcal{Q})$ the 2-sphere of complex structures $J \in \mathcal{Q}$  i.e.  \mbox{$S(\mathcal{Q})= \{aJ_1 + bJ_2 + cJ_3,\;  a,b,c \in \R, \; a^2 + b^2 + c^2=1 \}$.}

A real vector space $V$ endowed with a hypercomplex structure $(J_1,J_2,J_3)$ is an  $\mathbb{H}$-module  by defining scalar multiplication by a quaternion $q$ as follows:
\[
qX= (a+ib+jc+dk)X= aX + bJ_1X +cJ_2X+dJ_3X, \quad X \in V, \; a,b,c,d, \in \R
\]
and $(i,j,k)$ a basis of $Im(\mathbb H)$ satisfying
\begin{equation} \label{multiplication table of quaternions}
i^2=j^2=k^2=-1; \; ij=-ji=k.
\end{equation}

\begin{defi}
An Euclidean scalar product $< \, , \,>$ in $V$ is called \textbf{Hermitian} w.r.t. a hypercomplex basis  $\mathcal{H}= (J_\alpha)$ (resp. the quaternionic structure $\mathcal{Q}= span(\mathcal{H})_\R$) iff for any $X,Y \in V$
\[
<J_\alpha X, J_\alpha Y>=  < X,  Y> \text{or equivalently }  \; <J_\alpha X,  Y>=  -< X, J_\alpha Y>, \qquad (\alpha=1,2,3)
\]
(respectively
\[
<J X, J Y>=  < X,  Y> \text{or equivalently }  \; <J X,  Y>=  -< X,  JY>, \qquad (\forall J \in \mathcal{Q}) \text{)}\]
\end{defi}

\begin{defi}
  A hypercomplex structure  $\mathcal{H}$ (resp. quaternionic structure $\mathcal{Q}$) together with an Hermitian scalar product $< \, , \,>$ is called an \textbf{Hermitian hypercomplex} (resp. \textbf{Hermitian quaternionic}) \textbf{structure} on $V$  and the triple $(V^{4n}, \mathcal{H},<,>)$ (resp. $(V^{4n}, \mathcal{Q}, <, >)$) is  an \textbf{Hermitian  hypercomplex} (resp. \textbf{  quaternionic} ) \textbf{ vector space}.
\end{defi}

The prototype of an Hermitian  hypercomplex vector space is the $n$-dimensional quaternionic numerical space
$\mathbb{H}^n$ which
is a real vector space of dimension $4n$, a
$\mathbb{H}$-module with respect to  right (resp. left)
multiplication  by quaternions
and is endowed with the canonical positive definite Hermitian product
\begin{equation} \label{canonical Hermitian product in $H^n$}
\begin{array}{ll}
\mathbf{h} \cdot \mathbf{h'}= \sum_{\alpha=1}^{n} \overline{h_\alpha} h'_\alpha
& (\text{resp.} \; \mathbf{h} \cdot \mathbf{h'}= \sum_{\alpha=1}^{n} h_\alpha \overline{h'_\alpha})\\
& \quad \mathbf{h}=(h_1, \ldots,h_n),\; \mathbf{h'}=(h_1', \ldots,h_n') \in
 \mathbb{H}^{n} .
\end{array}
\end{equation}

 The real part of the Hermitian product defines an Euclidean   scalar product $< \, , \,>= Re(\cdot)$
on the real vector space  $\mathbb{H}^{n} \simeq \R^{4n}$.
If we consider the basis $(1,i,j,k)$ of  $\mathbb{H}$ satisfying the multiplication table obtainable from the conditions (\ref{multiplication table of quaternions})
 one has that the right multiplications by $-i, -j, -k$ (resp. left multiplication by $i,j,k$) induce real endomorphisms $(I=R_{-i},J=R_{-j},K=R_{-k})$
 (resp. $(I=L_{i},J=L_{j},K=L_{k})$) of the $\mathbb{H}$-module
$\mathbb{H}^n$  satisfying $I^2=J^2=K^2=-Id, \; IJ=K=-JI$
 and skew-symmetric with respect to the metric $< \, , \,>$  i.e. an Hermitian hypercomplex structure on $\mathbb{H}^{n}$.

We recall that
a new basis $(1,i',j',k')$ of $\mathbb{H}$ give rise to the  multiplication table (\ref{multiplication table of quaternions}) iff  $(i',j',k')=(i,j,k)C$ with $C \in SO(3)$.
 We will denote by $\mathcal{B}$ the set of bases of $\mathbb{H}$ satisfying the relations (\ref{multiplication table of quaternions})  and call it \textbf{canonical system of bases}.
In \cite {BR1} it has been proved that
\begin{prop}  \label{intrinsic meaning of Hermitian and scalar product in H^n}\cite{BR1} Both the Hermitian product and the scalar product of $\mathbb{H}^n$ have intrinsic meaning (SPIEGARE) with respect to the canonical system of bases $\mathcal{B}$.
\end{prop}

Let $(1,i,j,k) \in \mathcal{B}$ be a chosen basis in $\mathbb{H}$  and denote by $I=R_{-i},\; J=R_{-j}, \;K=R_{-k}$ the real endomorphisms of  the $\mathbb{H}$-module $\mathbb{H}^{n}$. Let $\mathcal{Q}= span_\R(I,J,K)$.
\begin{prop} For the scalar product and the Hermitian product of a pair of vectors $L,M \in \mathbb{H}^n$ the following relation holds:
\begin{equation} \label{relation between scalar product and hermitian product of H^n}
L \cdot M=<L,M> +<L,IM>i +<L,JM>j +<L,KM>k
\end{equation}
\end{prop}
\begin{proof}
We prove that $<L,IM>,<L,JM>,<L,KM>$ are respectively the coefficients of $i,j,k$ in the Hermitian product $L \cdot M$. In fact
$<L,IM>=Re(L \cdot -Mi)=  -Re(L \cdot M)i$ which is exactly the coefficient of $i$ of the quaternion  $L \cdot M$ and analogously for $<L,JM>$ and $<L,KM>$.
\end{proof}

After identifying an admissible  hypercomplex  structures  $(I,J,K)$ of $V$  with $(R_{-i},R_{-j},R_{-k})$ of $\mathbb{H}$, we can endow a quaternionic Hermitian vector space with the Hermitian product given in  (\ref{relation between scalar product and hermitian product of H^n}). It has an intrinsic meaning w.r.t. the admissible bases, that is,  using a different admissible basis $(I',J',K')=  (I,J,K)C, \;  \, C \in SO(3)$,  the coordinates of the obtained quaternion are w.r.t. the basis $(i',j',k')=(i,j,k)C$. In the following we consider such an identification.

\subsection{Special subspaces of an Hermitian quaternionic vector space}

 Let $(V^{4n}, \mathcal{Q}, <, >)$ be an Hermitian quaternionic vector space endowed with  the Hermitian product given in (\ref{relation between scalar product and hermitian product of H^n}). In the following, given a finite set $\{M_1, \ldots, M_s \}$ of vectors of $V$, we denote by $L(M_1, \ldots,M_s)$ or equivalently  by  $span_\R(M_1, \ldots, M_s)$ their linear span over  $\R$. We will denote by $(I,J,K)$ a generic admissible basis of $\mathcal{Q}$ and by $I^\perp=L(J,K)\cap S(\mathcal{Q})$ 
 i.e. $I^\perp= \{\beta J + \gamma K, \; \beta^2 + \gamma^2=1\}$. Moreover, given a subspace $U \subset V$,   $U^\mathbb{H}$  denotes its quaternionification  i.e. the subspace spanned  on $\mathbb{H}$ by some basis of $U$. In particular, given a vector $X \in V$ by $\mathcal{Q} X$ we denote the 4-dimensional real subspace real image of the 1-dimensional subspace generated by $X$ over $\mathbb{H}$ i.e. $\mathcal{Q} X= span_\R(X,IX,JX,KX)= (\R X)^\mathbb{H}$.

\begin{defi}
Let $(V^{4n}, \mathcal{Q}, <,>)$ be an Hermitian quaternionic vector space.
A subspace $U \subset V$ is \textbf{quaternionic} if $AU=U, \; \forall A \in S (\mathcal{Q})$.
A subspace is \textbf{pure} if it does not contain any non trivial quaternionic subspace.
Let $I \in S (\mathcal{Q})$, then $U$ is  \textbf{$I$-complex} and we denote it by $(U,I)$ if $U=IU$. In particular $(U,I)$ is a \textbf{totally $I$-complex} subspace  if, for any $J \in I^\perp$,  the pair of subspaces $(U,JU)$ are strictly orthogonal  \footnote{ A pair of subspaces $A,B$ are  \textit{orthogonal}   if the angle between them is $\pi/2$ i.e. if there exists a line in $A$ orthogonal to $B$. In particular they are  \textit{strictly orthogonal} and we wrote $A \perp B$ if any line in $A$ is orthogonal to $B$. In terms of principal angles (whose definition we recall in (\ref{definition of principal angles})), we can say that  a pair of subspaces $A,B$  of dimensions $m,n, \; m \leq n$ is orthogonal if at least one of the principal angle is $\pi/2$ and is strictly  orthogonal of all $m$ principal angles equal $\pi/2$. Clearly   the pair $(A,B)$ is strictly orthogonal iff if is orthogonal and isoclinic (see the Definition (\ref{definition of isoclinicity}). For instance, for any $T \in S(\mathcal{Q})$ and $U \subset V$ a 2-plane, the pair $U,TU$ is  isoclinic. Then in this case one can speak  indifferently of orthogonality or strictly orthogonality.}.
A subspace is \textbf{totally real} if it does not contain any complex subspace  or equivalently if $AU \cap U= \{0 \}, \, \forall A \in S(\mathcal{Q})$.   In particular it is a \textbf{r.h.p.s.} (real hermitian product subspace) if,  for any pair of vectors  $X,Y \in U$, one has $X \cdot Y \in \R$.  
\end{defi}
A totally $I$-complex subspace is   a \textbf{c.h.p.s.}  (complex hermitian product subspace) i.e.,  for any pair of vectors  $X,Y \in U$ , $X \cdot Y = a + ib, \; a,b \in \R$.
Clearly a 2-dimensional $I$-complex subspace is totally $I$-complex.  Furthermore a totally real subspace is a r.h.p.s. iff for any admissible basis,  the pairs $(U,IU), \,  (U,JU), \,  (U,KU)$ are strictly orthogonal.
We recall some results regarding the subspaces just defined:
\begin{claim} \label{properties of some quternionic,complex and real subspaces} $\,$
\begin{itemize}
\item For quaternionic subspaces:
\end{itemize}
\begin{enumerate}
\item \label{1} The sum and the intersection of quaternionic subspaces is a quaternionic subspace.
\begin{proof}
In fact if $0 \neq X \in U_1 \cap U_2$ with $U_1,U_2$ quaternionic subspaces  then $\mathcal{Q}X \in U_1$ and $\mathcal{Q}X \in U_2$ then  $\mathcal{Q}X \in U_1\cap U_2$. For the sum, one has that for every vector $Z=X+Y \in U_1 + U_2$  and every $A \in S(\mathcal{Q})$ one has $AZ=AX+AY \in U_1 + U_2$.
\end{proof}
\item \label{2} The orthogonal complement of a quaternionic subspace of a quaternionic space is quaternionic.
\begin{proof}
Let $U$ be quaternionic and $W \subset U$ quaternionic as well. Let consider the orthogonal decomposition $U=W  \oplus W^\perp$. For $X \in W^\perp, \; A \in S(\mathcal{Q})$, let
$AX=X_1 + X_2$ with $X_1 \in W=AW$ and $X_2 \in W^\perp$. Being $A(AX)=-X=AX_1 + AX_2$ it follows that $X_1=\textbf{0}$ and any $A \in S(\mathcal{Q})$ preserves $W^\perp$ i.e. $W^\perp$ is quaternionic.
\end{proof}

\begin{itemize}
\item For $I$-complex subspaces:
\end{itemize}
\item The subspace $U$ is $I$-complex iff it is $(-I)$-complex.  In the following, when we will speak of an $I$-complex subspace we will always imply "up to sign".
\item \label{A complex subspace does not contain a complex subspace by a different complex structure} Let $(U,I)$ be pure and suppose $(U',I') \subseteq U$. Then $I'=\pm I$. In other words an $I$-complex subspace does not contain any complex subspace by a different  complex structure (up to sign).
\begin{proof}
Suppose $0 \neq Y \in U \cap U'$ then $ I(I'Y)= I (\alpha IY + \beta JY + \gamma KY)= -\alpha Y + \beta KY - \gamma JY \in U$ which implies $\beta KY - \gamma JY= K(\beta Y - \gamma IY) \in U$. From pureness of $U$ it follows  $\beta=\gamma=0$ i.e. $I'=\pm I$.
\end{proof}
\item   \label{3} The orthogonal complement to a complex subspace $W$ of an $I$-complex space $U$ is an $I$-complex subspace.
\begin{proof}
From previous statement,  $W$ is necessarily $\pm I$-complex. Consider then  $(W,I) \subset (U,I)$ and the orthogonal decomposition $(U,I)= (W,I) \stackrel{\perp} \oplus \tilde U$ where $\tilde U$ is the orthogonal complement to $W$ in $U$.
Let $Z \in \tilde U$ non null and consider the vector $IZ=X + Y$ with $X \in W$ and $Y \in \tilde U$. Then $\tilde U \ni I(IZ)=IX + IY $ implies that $X=0$ i.e. $\tilde U$ is $I$-complex.
\end{proof}

\item Sum and intersection of $I$-complex subspaces is an $I$-complex subspace.
\begin{proof}

Let $W=(U_1,I) \cap (U_2,I)$. If $Z \in W$ then $IZ \in U_1$ and $IZ \in U_2$ then $W$ is $I$-complex. Let now consider $U_1 + U_2= \bar{U_1} \stackrel{\perp} \oplus W \stackrel{\perp} \oplus \bar{U_2}$ where $\bar{U_1}$  (resp.  $\bar{U_2}$) is  the orthogonal complement to $W$ in $U_1$ (in $U_2$).
From (\ref{3}), the subspaces $\bar{U_1}$ and $\bar{U_2}$ are $I$-complex.
If $T \in (U_1 + U_2)= X + Y + Z $ with $X \in \bar{U_1}, \; Y \in W,\; Z \in \bar{U_2}$ one has $IT \in U_1 + U_2$.
\end{proof}

\item  \label{for any K in I perp  KU=JU is I-complex} Let $(U,I)$ be an $I$-complex subspace. An \textbf{adapted basis} of $U$ is an admissible basis containing  $I$. For any $K,K' \in I^\perp$ one has that $KU=K'U$ is $I$-complex.
 \begin{proof}
  Let $(I,J,K)$ be an adapted bases. One has  $JU=KIU=KU$. Let $K'= \alpha J + \beta K  \in I^\perp$. 
  For any  $X \in U$ it is $K' X= \alpha J X + \beta K X \in J U$. The subspace $KU$ is clearly $I$-complex since $IKU=KIU=KU$.

 \end{proof}
\item  If $(U,I)$ is pure, then $JU \cap U= \{0 \}$ for any $J \in I^\perp$.
\begin{proof}
In fact suppose  $JY \in U \cap JU$ with $Y \in U$. Then $Y,IY,JY$ are in $U$ as well as  $I(JY)=KY$ which is absurd by the hypothesis of pureness.
\end{proof}
 \item \label{intersection} The intersection of a pair of pure complex subspaces by different, up to sign, complex structures  is a totally real subspace.
  \begin{proof}
Let $(U,I)$ and $(U',I')$ be a pair of pure complex subspaces with $I' \neq \pm I$ and denote $W=U \cap U'$. Suppose $W$ is complex. Applying the previous result, $W$ as a subspace of $U$  can only be a pure $\pm I$-complex and as a subspace of $U'$ can only be  pure $\pm I'$-complex. Then $W$ is totally real.
\end{proof}
\begin{itemize}
\item For totally real subspaces:
\end{itemize}
\item By definition, a totally real subspace  $U$ is pure.
\item \label{9} Given a hypercomplex basis $(I,J,K)$ it is $IU \cap JU= IU \cap KU= JU \cap KU=\{0 \}$.
\begin{proof}
In fact, suppose $0 \neq Y= Iv=Jw \in IU \cap JU$ for the non null vectors, $v,w \in U$. Then  $I^2v= -v= IJ(w)=Kw$ which is absurd since by definition $U \cap KU= \{0\}$.
\end{proof}
 \item Let $U$ be a totally real subspace.
 If $U$ is  a r.h.p. subspace   $\dim (U)^\mathbb{H} =4 \dim U$ otherwise   $\dim (U)^\mathbb{H} \geq 2 \dim U$.
 \begin{proof}
If $U$ is  a r.h.p.s. for any admissible basis $(I,J,K)$,  any pair of the 4 subspaces $(U,IU,JU,KU)$ is strictly orthogonal.
If instead $U$ is totally real  
 then clearly $\dim (U)^\mathbb{H} \geq 2 \dim (U)$ since, from (\ref{9}), any pair of the subspaces $U,IU,JU,KU$ have trivial intersection.
 \end{proof}

\item Any  subspace of a r.h.p. subspace is a r.h.p.s.

\begin{proof}
 It's straightforward.
\end{proof}

\end{enumerate}
\end{claim}

\subsection{Decomposition of a generic subspace} \label{Decomposition}

The following proposition shows that, by using quaternionic, pure complex
and totally real subspaces as building blocks, we can build up any subspace $U$ of  an Hermitian quaternionic vector space  $(V^{4n}, \mathcal{Q}, <, >)$).

\begin{prop} \label{Decomposition of a subspace}
Let $U \subseteq V$ be a subspace  and let $U_{Q}$ be its  maximal quaternionic subspace. Then $U$ admits an orthogonal   decomposition of the form
\begin{equation}\label{Decomposition of a generic subspace}
\begin{array}{l}
U^m= U_{Q} \stackrel{\perp} \oplus U^\Sigma \stackrel{\perp} \oplus U_R    \hskip 2.0cm \text{ with} \\
U^\Sigma=(U_1,I_1) \stackrel{\perp} \oplus \ldots \stackrel{\perp} \oplus (U_p,I_p)
\end{array}
\end{equation}
where
$(U_i,I_i)$ are maximal pure $I_i$-complex and $U_R$ is totally real.
\end{prop}

\begin{proof}
For any $A \in S(\mathcal{Q})$ we denote by $U_A$ the maximal  $A$-invariant subspace in $U$.
Let $U^1$ be the orthogonal complement to  $U_{Q}$ in $U$  and choose a complex structure $I_1$ such that $(U_1,I_1):=U^1_{I_1} \neq \{0 \}$. Then we can write $U=
U_{Q} \stackrel{\perp} \oplus (U_1,I_1) \stackrel{\perp} \oplus U^2$ where $U^2$ is the orthogonal complement in $U$  to  $U_{Q} \stackrel{\perp} \oplus (U_1,I_1)$.

Let now choose a complex structure $I_2$ such that $(U_2,I_2):=U^2_{I_2} \neq \{0 \}$. Then  $U=
U_{Q} \stackrel{\perp} \oplus (U_1,I_1) \stackrel{\perp} \oplus (U_2,I_2) \stackrel{\perp} \oplus U^3$ where $U^3$ is the orthogonal complement to  $U_{Q} \stackrel{\perp} \oplus (U_1,I_1) \stackrel{\perp} \oplus (U_2,I_2)$.

Denote by $p+1$ the step in which  $U^{p+1}$ has no invariant complex subspace. Then  $U^{p+1}= U_R$ is totally real.
\end{proof}

\begin{prop}
It is $(U_R)^\mathbb{H} \cap U_Q= \{ 0\}$.
\end{prop}
\begin{proof}
Suppose $W=(U_R)^\mathbb{H} \cap U_Q$. From point (\ref{1}) of the Claim (\ref{properties of some quternionic,complex and real subspaces}), $W$ is quaternionic.  Let $(U_R)^\mathbb{H}= W \oplus W^\perp$. From point (\ref{2}) of the same Claim, $W^\perp$ is quaternionic and since $U_R \subset W^\perp$ being by construction $U_Q \perp U_R$ one has that $(U_R)^\mathbb{H}= W^\perp$ i.e. $W= \{\textbf{0}\}$.
\end{proof}

In the following, given a pair of subspaces $(A,B)$, by $A \stackrel{H} \perp B$ we intend that $A$ and $B$ are  orthogonal in Hermitian sense i.e. their quaternionifications $A^\mathbb{H},B^\mathbb{H}$ are strictly orthogonal (in other words $A \stackrel{H} \perp B$ is equivalent to $A^\mathbb{H} \perp B^\mathbb{H}$).
We now prove the following facts:
\begin{enumerate}
\item In the orthogonal sum $U_{Q} \stackrel{\perp} \oplus (U_1,I_1) \stackrel{\perp} \oplus \ldots \stackrel{\perp} \oplus (U_p,I_p)$ any pair of the complex addends    are orthogonal in Hermitian sense. 
\item If in (\ref{Decomposition of a subspace}) one has $U_R= \{0 \}$ the given decomposition is unique.
\end{enumerate}

To prove (1) we need the following
\begin{lemma} \label{orthogonal complex 2-planes by different complex structures are Hermitian othogonal}
Let $U_1=L(X,IX)$  and $U_2=L(Y,I'Y)$ be a pair of 2-dimensional complex subspaces .
Then if $I' \neq  \pm I$ we have that
\[
U_1 \perp U_2 \Leftrightarrow U_1 \stackrel{H} \perp U_2, \quad \text{( i.e.} \Leftrightarrow    \mathcal{Q} X  \perp \mathcal{Q} Y \text{)}.
\]
\end{lemma}

\begin{proof}
Since for any complex structure  $A \in  S(\mathcal{Q})$, the 2-plane $L(X,AX) \subset  \mathcal{Q} X$,
then clearly if $ \mathcal{Q} X \perp  \mathcal{Q} Y$, we have $U_1 \perp U_2$.  Viceversa let suppose that $I$ is a complex structure and $(I,J,K)$ an adapted basis.
Let $I'= \alpha I + \beta J + \gamma K$, with $\alpha^2 + \beta^2 + \gamma^2= 1$. Then $U_1 \perp U_2$ if
\[
\begin{array}{l}
0= <Y,X>\\
0= <Y,IX>\\
0=<I'Y,X>= <\alpha IY + \beta JY + \gamma KY,X> \Rightarrow  -\beta <Y,JX> - \gamma <Y,KX>=0\\
0=<I'Y,IX>= <\alpha IY + \beta JY + \gamma KY,IX> \Rightarrow  -\gamma <Y,JX> +  \beta <Y,KX>=0,
\end{array}
\]
This implies
\[
\left\{
\begin{array}{l}
 Y \perp U_1\\
 (-\beta^2 - \gamma^2) <Y,JX>=0.
\end{array}
\right.
\]
 Then
\begin{equation} \label{condition for existence of orthogonal complex}
U_2 \perp U_1 \Leftrightarrow  \left\{
\begin{array}{l}
I'= \pm I \; \; \text{and} \; \; Y \perp U_1 \qquad  \text{ or} \\
Y \perp  \mathcal{Q} X 
\end{array}
\right.
\end{equation}
since from above  $<Y,JX>=0 \Rightarrow <Y,KX>=0$.

Then a pair of complex  2-planes $U_1=L(X,IX), \; U_2=L(Y,I'Y)$  are orthogonal iff  whether $I=\pm I'$ and $Y \perp U_1$
 (but not necessarily to $ \mathcal{Q} X$)  or for the  quaternionic subspaces $ \mathcal{Q} X \perp  \mathcal{Q} Y$. \footnote{This Lemma is true also  in a para-quaternionic Hermitian vector space. In that case the Lemma applies not only to a pair of 2-dimensional complex subspaces  but also to a pair of 2-dimensional para-complex subspaces or  to a pair made of a 2-dimensional complex and a 2-dimensional para-complex subspace. For interested readers, in \cite {MV} and \cite{MV2} it is possible to find the  analogue decomposition  of a para-quaternionic Hermitian vector space which differs from (\ref{Decomposition of a generic subspace}) because of the existence, in that case, of (weakly) para-complex and nilpotent subspaces.}
\end{proof}

It follows that if 
$Y \in \mathcal{Q} X $ then  $L(Y,\tilde I Y)$ is never orthogonal to $U_1$ unless $\tilde I=\pm I$
(in which case $L(Y, \pm I Y)= L(JX,KX)$, i.e.
a pair of complex 2-planes by different (up to sign) complex structures and belonging to  the same quaternionic line are never orthogonal to each other.
Viceversa, since  any 2-plane in a quaternionic line is $\tilde I$-complex for some $\tilde I \in S(\mathcal{Q})$, if $U_1$ and $U_2$ are orthogonal 2-planes belonging to  the same quaternionic line  then they are complex by the same complex structure. We can now state the
\begin{prop} \label{orthogonal complex subspaces by different structures}
Strictly orthogonal complex subspaces by different complex structures are orthogonal in Hermitian sense i.e. their quaternionifications  are strictly orthogonal in pair. In particular the different complex addends in $U^\Sigma$ given in (\ref{Decomposition of a generic subspace}) are orthogonal in Hermitian sense.
\end{prop}

\begin{proof}
Let consider the pair $(U_1,I_1)$ and $(U_2,I_2)$, $I_1 \neq \pm I_2$, of strictly orthogonal complex subspaces. Any $I$-complex subspace can be decomposed into the orthogonal sum of $I$-complex 2 planes. From (\ref{orthogonal complex 2-planes by different complex structures are Hermitian othogonal}), any 2-plane of the decomposition of $(U_1,I_1)$ is Hermitian orthogonal to any 2-plane of the decomposition of $(U_2,I_2)$ then ${U_1}^\mathbb{H} \perp {U_2}^\mathbb{H}$. 
We conclude  then the different addends of  $U^\Sigma$  belong to strictly  orthogonal quaternionic subspaces.
\end{proof}

\begin{defi}
We call \textbf{$\Sigma$-complex subspace} a pure subspace $ U \subset (V^{4n}, \mathcal{Q}, <, >)$   orthogonal sum of maximal complex subspaces $(U_i,I_i)$ with $I_i \in S(\mathcal{Q})$. We denote it by $(U, \mathcal{I}$)  where $\mathcal{I}= \{I_i\}$.
\end{defi}
Fixed an admissible basis $(I,J,K)$, we need to order the set  $\mathcal{I}= \{I_i= \alpha_i I + \beta_i J + \gamma_i K\}$. A way to do it is for instance by  using  the lexicographic 
order of the coefficients $\alpha_i,\beta_i,\gamma_i$. 
By Proposition (\ref{orthogonal complex subspaces by different structures}), the complex addends of  a $\Sigma$-complex subspace  are  orthogonal in Hermitian sense.

To prove the following proposition concerning the decomposition of  a $\Sigma$-complex subspace, we need the
\begin{lemma} \label{the direct sum of a pair of complex subspaces by different structure contains no other complex subspace}
The orthogonal  sum of a pair of  pure complex subspaces $(U_1,I)$, $(U_2,I'), \, I \neq \pm I'$  contains no other 
complex subspace (not contained in $U_1$ or $U_2)$.
\end{lemma}

\begin{proof}
Let $(U_1,I)$ and $(U_2,I'), \; I \neq \pm I'$ be a pair of orthogonal  pure complex subspaces and consider their sum $U=U_1 \stackrel{\perp} \oplus U_2$.  From Proposition (\ref{orthogonal complex subspaces by different structures}) one has that $U_1^\mathbb{H} \perp U_2^\mathbb{H}$. Suppose there exists  $0 \neq T \in U$,  $T= X+Y, \; X \in U_1, \, Y \in U_2$ and   $\widetilde {I}=\alpha I +\beta J + \gamma K  \in S (\mathcal{Q})$ such that  $\widetilde {I} T \in U$.
The vector $\widetilde {I} T= (\alpha IX + \beta JX + \gamma KX) + (\alpha IY + \beta JY + \gamma KY)$  is orthogonal sum (in a unique way) of a vector  $T_1 \in U_1$ and a vector  $T_2 \in U_2$. 
 The vector $\alpha IX \in U_1$ whereas  the vector $\beta JX + \gamma KX \in U_1^\perp$ and also $\beta JX + \gamma KX  \perp (\alpha IY + \beta JY + \gamma KY) \in QY$. Then necessarily $\beta JX + \gamma KX=0$ i.e. $I(JX)= \frac{-\beta}{\gamma} JX$. Since $I$ has no real eigenvalues, such equation is satisfied if $X=0$ 
 or if $\beta=\gamma=0$ i.e. if $\widetilde {I}= \pm I$. Analogously, if $\widetilde {I}=\alpha' I' +\beta' J' + \gamma' K'$ one has that  $\beta' J' Y + \gamma' K'Y=0$ which implies that  $Y=0$  or  $\beta'=\gamma'=0$ i.e. if $\widetilde {I}= \pm I'$.
Then excluding the case $X=Y=0$ the only possibilities are $\tilde I=\pm I$ and $Y=0$ or  $\tilde I=\pm I'$ and $X=0$. The first implies that $T=X \in U_1$ and $\tilde T= IX \in U_1$; the second implies $T=Y \in U_2$ and $\tilde T=I'Y \in U_2$.
\end{proof}

 Extending the previous proof to a finite number of addends we can then state the following
\begin{prop} \label{unicity of the decomposition of a subspace with trivial real addend}
Let $U \subset V$ be a $\Sigma$-complex  subspace. Then all other subspaces in $U$ (not contained in the complex addends)  are totally real. In other words the given decomposition of $U$ into an Hermitian  orthogonal sum of maximal pure complex subspaces by different structure (up to sign) is unique.
\end{prop}

This results applies in particular to the $\Sigma$-complex subspace $U^\Sigma$ of the decomposition given in (\ref{Decomposition of a generic subspace}). 

\vskip 0.3cm
We finally underline  that in general \textit{the Euclidean orthogonal sum $U^\Sigma \stackrel{\perp} \oplus U_R$ in (\ref{Decomposition of a subspace}) is \underline{not} orthogonal in Hermitian sense}. Moreover
the two addends are not canonically defined as can be easily seen by considering  for example the 4-dimensional pure subspace $U= (U_1,I_1) \oplus (U_2,I_2),  \; I_1 \neq I_2$  direct but not orthogonal sum of a pair of (totally) complex 2-planes. The decomposition $U= U_1 \stackrel{\perp} \oplus (U_1^\perp \cap U)$  and $U_2 \stackrel{\perp} \oplus (U_2^\perp \cap U)$ are two different orthogonal decomposition where $U_1^\perp \cap U$ and   $U_2^\perp \cap U$ are different totally real 2-planes.

\section{Preliminaries}

 We define the  (Euclidean) angle between two subspaces of dimension $p$ and $q$ of an Euclidean vector space $E^n$ by using exterior algebra (see \cite{J} among others).
 Let $(E^n, \, <\,,\,> )$ be an $n$-dimensional vector space endowed with an Euclidean scalar product.
 Any decomposable $p$-vector $\alpha= a_1 \wedge \ldots \wedge a_p \in \Lambda^p E^n$ corresponds to an oriented subspace $A^p \in E^n$ and precisely to the one spanned by $a_1,\ldots, a_p$. Conversely, for any basis  of $A^p$ the wedge of these vectors is a multiple of $\alpha$ (i.e. it is equal to $k \alpha$, with $k \in \R, k \neq 0$).
The scalar product  $<\,,\,>$ in $E^n$ induces  a scalar product  $ \; \cdot \; $    in the vector space  $\Lambda^p E^n$  by defining
\[
 \alpha \cdot \beta= det (<a_i, b_j>)
\]
for a pair of decomposable vectors $\alpha = a_1 \wedge \ldots, \wedge a_p; \qquad   \beta = b_1 \wedge \ldots, \wedge b_p, \quad a_i, b_i \in E^n$
and then extending for linearity to any pair of vectors of $\Lambda^p E^n$.

It is  definite positive and non degenerate then the pair $(\Lambda^p E^n, <>)$ is an Euclidean vector space.
In particular for the angle between $\alpha$ and $\beta$,
\begin{equation} \label{angles between p-vectors}
\cos \widehat{\alpha \beta}= \frac{\alpha \cdot \beta}{\sqrt{\alpha \cdot \alpha} \;  \sqrt{\beta \cdot \beta}}=\frac{\det(<a_i,b_i>)}{mis \; \alpha  \; mis \; \beta}.
\end{equation}

being
 \[mis \; \alpha= |\alpha|=  \sqrt{\alpha \cdot \alpha}.\]

 Given $A^p$ and $B^q$ and $\alpha= a_1 \wedge \ldots, \wedge a_p \in \Lambda^p E^n$ associated to $A$ and $\beta= b_1 \wedge \ldots, \wedge b_q \in \Lambda^q E^n$ associated to $B$, we consider the orthogonal projections of $a_1,\ldots, a_p$ on $B$ and $B^\perp$.
 Then $a_i= a_i^H + a_i^V$, and  $\alpha= \alpha_H + \alpha_V + \alpha_M$ (where $M$ stands for mixed part).

If we choose another basis in $A$ (then  $\alpha'=k \alpha$) we have \[ \alpha'_H= k \alpha_H, \quad    \alpha'_V= k \alpha_V, \quad \alpha'_M= k \alpha_M. \]

\begin{defi} \label{angle between non oriented subspaces}
The angle $\widehat{A,B}$ between  the non oriented subspaces $A^p$ and $B^q$, $p \leq q$ is the usual angle (between two lines, a line and a plane, two planes) i.e. the angle between
 one subspace and its orthogonal projection onto the other i.e.
\[  \theta= \arccos \frac{|\alpha_H|}{|\alpha|}.\]
\end{defi}

Then $\theta \in [0, \pi/2]$ and, from previous Lemma, it is independent from the chosen basis in $A$. In particular, if $p=q$   we can write
\begin{equation} \label{angle between subspaces of same dimension}
  \theta= \arccos \frac{|\det(<a_i, b_j>)|}{|\alpha| \; | \beta|}
\end{equation}
 i.e. \textit{the cosine of the angle between a pair of $p$-planes $A,B \subset E^n$
 equals the absolute value of the cosine of the angle between  any pair of $p$-vectors $\alpha,\beta  \in \Lambda^p E^n$ corresponding to $A$ and $B$}.

\begin{remk} \label{angle between oriented subspaces}
In case we consider oriented subspaces of  $E^n$  then we do not take the absolute value in the previous expressions and one has  $\theta \in [0, \pi]$. Unless expressively stated, no orientation will be defined on  the subspaces we consider in this paper.
\end{remk}

 We recall  the definition of the principal angles between a pair of subspaces of a real vector space $V$ (see \cite{Gala},\cite{Riz1} among others).
\begin{defi} \label{definition of principal angles}  Let $A,B \subseteq V$ be subspaces, \mbox{$\dim k= dim(A) \leq dim(B)=l \geq 1$.} The \textbf{principal angles} $\theta_i \in [0.\pi/2]$ between the subspaces $A$ and $B$ are recursively defined for $i= 1, \ldots,k$   by
\[\cos \theta_i=<a_i,b_i>= \max \{ <a,b> \, : ||a||=||b||=1, \; a \perp a_m, \, b \perp b_m, \,  m= 1,2, \ldots, i-1 \}.
\]
The unitary vectors $\{a_j \}, \{b_j \}, j=1, \ldots,k$  are called the \textbf{principal vectors} of the pair $(A,B)$, in particular
 $(a_j, b_j) \in (A \times  B), j=1, \ldots,k$ are  \underline{related}    principal vectors corresponding to $\theta_j$.
\end{defi}
In words, the procedure is to find the unit vector $a_1 \in A$ and  the unit vector $b_1 \in B$ which minimize the angle between them and call this angle $\theta_1$.
Then consider the orthogonal complement in $A$ to $a_1$   and the orthogonal complement in $B$ to $b_1$ and iterate.
The principal angles $ \theta_1, \ldots, \theta_k$ between the pair of subspaces $A,B$ are  some of the critical values of the angular function
\[ \phi_{A,B}= A \times B \rightarrow \R\]
associating with each pair of non-zero vectors $a \in A, \; b \in B$ the angle between them.
In the following, given a pair of subspaces $A,B$ we will denote by $Pr_A^B:A \rightarrow B$ (resp. $Pr_B^A$) the orthogonal projector of $A$ onto $B$ (resp. $B$ onto $A$).
The principal angles are the diagonal entries of the orthogonal projector $Pr_B^A$  stated in the theorem  of Afriat (\cite{GS}, \cite{Af}):
\begin{teor} \label{Afriat Theorem} \cite{Af}, \cite{GS}.
In any pair of subspaces $A^{k}$ and $B^{l}$ there exist orthonormal bases $\{u_i\}_{i=1}^{k}$ and $\{v_j\}_{j=1}^{l}$
such that $<u_i , v_i> \geq 0$ and $<v_i, v_j>=0$ if $i \neq j.$
\end{teor}
\begin{proof}
It is a direct consequence of the following
\begin{lemma} \label{principal angles are singular values of projector}
Given finite dimensional subspaces $A,B$, let $a_1,b_1$ attain
\[
\max \{<a,b>, \quad a \in A,  \; b \in B, \quad ||a||=1, \; ||b||=1\}
\]
(i.e. the pair $(a_1,b_1)$ are the first pair of related principal vectors). Then
\begin{enumerate}
\item for some $\alpha \geq 0$,
\[Pr_A^B a_1= \alpha b_1, \qquad Pr_B^A b_1=\alpha a_1\]
\item  $a_1  \perp (b_1^\perp \cap B)$ and  $b_1  \perp (a_1^\perp \cap A)$ which leads the diagonal form of the matrix of the Projector $Pr_A^B$ (and $Pr_B^A$).
\end{enumerate}
\end{lemma}
To see that 1) holds, note that $Pr_A^B a_1= \alpha b$ where $\alpha, b$ minimize $||a_1-\alpha b||^2$ for $b \in B, ||b||=1$  and $\alpha$ a scalar. Thus to  minimize   $||a_1-\alpha b||^2= \alpha^2 -2 \alpha <a_1,b> +1$  we must maximize $<a_1,b>$. Moreover $\alpha=<a_1,b_1>$ is the cosine of the first principal angle.

For 2), let $A_1=a_1^\perp \cap A$ (resp. $B_1=b_1^\perp \cap B$). If $a \in A_1$, then $a \perp b_1$ since $ <a,b_1>= <Pr_A^A a,b_1>=<a, P_A^A b_1>=<a,\alpha a_1>= 0$.  Likewise if $b \in B_1$ then $b \perp a_1$. We proceed letting $a_2$ and $b_2$ attain
\[
\max \{<a,b>, \quad a \in A_1,  \; b \in B_1, \quad ||a||=1, \; ||b||=1\}
\]
and continue till we have exhausted $A$ and $B$.
\end{proof}

From (1) of (\ref{principal angles are singular values of projector}), one has that the cosines of the principal angles between the pair of subspaces $A,B$ of $V$ can  also defined as the singular values of the orthogonal projector $Pr_B^A$ (or equivalently  $Pr_A^B$).
If  $\alpha$ is a singular value, we call the pair $(a,b) \in (A \times B)$ such that  $Pr_A^B a= \alpha b, \qquad Pr_B^A b=\alpha a$ \textit{related} singular vectors   (associated to $\alpha$).

We underline the following relation between the  angle  and the principal angles between a pair of subspaces of a real vector space $V$ (see \cite{Vac}).
\begin{prop} \cite{SC} \label{cosine of angle between subspaces is the product of the cosines of the principal angles}
Let $A^p$ and $B^q$ be a pair of subspaces of $V^n$ with $1 \leq p \leq q \leq n$. Let $\theta$ be the angle between the subspaces $A^p$ and $B^q$  and   $\theta_1, \ldots, \theta_p$ the set of principal angles. Then
\[
\cos \theta= \cos \theta_1  \cdot \cos \theta_2 \cdot  \ldots \cdot \cos \theta_p.
\]
\end{prop}
In particular, if $p=q$, one has the well known result $|\det(G)|=  \cos \theta_1 \cdot \cos \theta_2 \cdot \ldots \cdot \cos \theta_p$ where $G$ is the matrix representing the Projector $Pr_A^B$ w.r.t. some orthonormal pair of bases. 

Let  recall the definition and some properties of isoclinic subspaces.
\begin{defi} \label{definition of isoclinicity}
A pair of non oriented subspaces $A$ and $B$ of same dimension are said to be \textbf{isoclinic} and the angle $\phi$  ($0 \leq \phi \leq \frac{\pi}{2}$)
is said to be angle of isoclinicity  between them  if either of the following conditions hold:\\
1) the angle between any non-zero vector of one of the subspaces and the other subspace is equal to $\phi$;\\
2) $G G^t= \cos^2 \phi \, Id$ for the matrix $G=<a_i,b_j>$ of the orthogonal projector $P_B^A: B \rightarrow A$  with respect to  any orthonormal basis $\{a_i\}$ of $A$ and $\{b_j\}$ of $B$;\\
3) all principal angles between $A$ and $B$ equal $\phi$.
\end{defi}

\begin{defi}
We denote by $\mathcal{IC}^{2m}$ the set of $2m$-dimensional subspaces of $V$ such that,  for any $A \in S(\mathcal{Q})$,   the pair $(U,AU)$ is isoclinic. When we do not need to specify the dimension we just use the notation $\mathcal{IC}$ and we call them simply \textbf{isoclinic subspaces}.
\end{defi}
 The fact that we consider only  even dimensions subspaces follows from the
\begin{prop}
 Let $U$ be an  odd dimension isoclinic subspace. Then $U$  is a real hermitian product subspace (r.h.p.s.). Namely
    $\mathcal{IC}^{2m+1}$ is the set of all and only the real Hermitian product $(2m+1)$-dimension subspaces.
 \end{prop}
   \begin{proof}
   If $U$ is a r.h.p.s. then by definition it  is  an isoclinic subspace. Viceversa $U$ is isoclinic and  $\dim U=2m +1$,  for any $A \in S(\mathcal{Q})$, by the skew-symmetry of $\omega^A$ one (and then all) principal angle is necessarily equal to $\pi/2$.

   \end{proof}

   Fixed an  admissible basis $(I,J,K)$  and, given $U \in \mathcal{IC}^{2m}$, we denote by $\theta^I, \theta^J, \theta^K$ the respective angles of isoclinicity.
In \cite{Vac_isoclinic} we introduced the following definitions. If  the pair $(U,IU)$  (resp. $(U,JU)$, resp. $(U,KU)$)  is strictly orthogonal (i.e. if all principal angles are $\pi/2$) we said that $U$ is $I$-orthogonal (resp. $J$-orthogonal, resp. $K$-orthogonal) and in general we spoke of single orthogonality (or 1-orthogonality).  When  two  (resp. three) of the above pair are strictly orthogonal we spoke of double  (resp. triple)-orthogonality. By saying that $U$ is \textbf{orthogonal} (without specifying the complex structures) we mean that at least one  among $\theta^I,\theta^J,\theta^K$ equals $\pi/2$.  Observe that only r.h.p. subspaces have a triple orthogonality.
In particular in this paper, the isoclinic subspaces we consider have no orthogonality unless they are totally complex in which case they have a double orthogonality.

Let $U$ be a subspace. Let fix an admissible basis $\mathcal{H}$ of $\mathcal{Q}$, and let   $f: U \times U \times \ldots \times U \rightarrow W$  some function
 where the codomain $W$ is some vector space. If $f$  is  constant  on its domain, we  will say that $f$ is an  \textit{invariant}  of $U$. If furthermore  the invariant $f$ does not depend on the chosen hypercomplex basis $\mathcal{H}$, we will say that $f$ is  an \textit{intrinsic property}  of $U$.

 Finally we recall the notion of  K\"{a}hler angle which is defined in a real  vector space $V$ endowed with a complex structure $I$.
 \begin{defi} Let $(V^{2n},I)$ be a real vector space endowed with a complex structure $I$.
 For any pairs of non parallel vectors  $X,Y \in V$ their  \textbf{K\"{a}hler angle} is given by
 \begin{equation} \label{definition of Kaehler angle}
 \Theta^I=\arccos \frac{<X,IY>}{|X| \, |Y| \sin \widehat{XY}}=\arccos \frac{<X,IY>}{mis \; (X \wedge Y)}.
 \end{equation}
  \end{defi}
Then $0 \leq \Theta^I \leq \pi$.
 It is straightforward to check that the K\"{a}hler angle is an intrinsic property of the oriented 2-plane $U=L(X,Y)$.  For this reason one  speaks of the K\"{a}hler angle  of a 2-plane. The K\"{a}hler angle measures the deviation of a 2-plane from holomorphicity. For instance the K\"{a}hler angle of a  r.h.p. subspace $U$   is  $\Theta^I(U)=\pi/2$  and  the one of a complex plane $(U,I)$ 
  is  $\Theta^I(U) \in \{0, \pi\}$.

The cosine of the K\"{a}hler angle of the pair of 2-planes  with opposite orientation $U$ and $\tilde U=L(Y,X)$  have opposite sign i.e. $\cos \Theta^I(U)=-\cos \Theta^I(\tilde U)$,
then, if one disregards the orientation of the 2-plane $U$, we can consider the absolute value of the right hand side of equation (\ref{definition of Kaehler angle})   restricting the K\"{a}hler angle to the interval $[0,\pi/2]$. In this case   the K\"{a}hler angle of the 2-plane $U$ coincides with one of the two identical principal angles, say $\theta^I(U)$, between the pairs of 2-plane $U$ and $IU$ (same as the ones of the pair  $(\tilde U,I \tilde U)$) which  are always isoclinic 
as one can immediately verify, then
\[\cos \theta^I(U)=|\cos \Theta^I(U)|\]
and from  (\ref{cosine of angle between subspaces is the product of the cosines of the principal angles}), one has
 \begin{equation} \label{angle between 2-planes U,IU}
   \cos (\widehat{U,IU})= \cos (\widehat{\tilde U,I \tilde U})= \cos^2 \theta^I(U) =  \cos^2 \Theta^I(U)=  \frac{<X,IY>^2}{mis^2 \; (X \wedge Y)}.
\end{equation}

Generalizing  the notion of K\"{a}hler angle,   in an Hermitian quaternionic vector space  $(V^{4n}, \mathcal{Q}, <,>)$ we will  speak of the  \textbf{$A$-K\"{a}hler angle}  of  an oriented  2-plane $U$ with   $A \in S(\mathcal{Q})$ and denote  it by $ \Theta^A(U)$.

In \cite{Vac_isoclinic} we proved the following
\begin{prop} \label{invariance of the sum of the 3 angles (U,IU),(U,JU),(U,KU)}
Let $U \subset (V^{4n}, \mathcal{Q},<,>)$ be a 2 plane. The sum of the cosines of the  angles between the pairs $(U,IU)$, $(U,JU)$, $(U,KU)$  is constant for any admissible basis $(I,J,K)$ of $\mathcal{Q}$.
\end{prop}

It follows that the quantity 
\begin{equation} \label{sum of square cosines of I,J,K-kaehòer angles of a 2-plane}
\cos^2 \Theta^I(U) + \cos^2\Theta^J(U)+ \cos^2 \Theta^K(U)= \cos^2 \theta^I(U) + \cos^2\theta^J(U)+ \cos^2 \theta^K(U)
\end{equation}
 is an intrinsic property of a 2-plane. 

Given a subspace $U$ of an Hermitian quaternionic vector space and generalizing a well known  notion relative to an Hermitian  complex vector space, for any $A \in S(\mathcal{Q})$ we will call  \textbf{$A$-K\"{a}hler form} of   $U$  the  skew-symmetric bilinear form
\[
\begin{array}{llll}
\omega^A|_U: & U \times U & \rightarrow & \R\\
 & (X,Y)& \mapsto &  <X,AY>.
 \end{array}
 \]

It is well known that the $A$-K\"{a}hler form admits a \textit{standard form}, namely  w.r.t. an orthonormal \textit{standard basis} $(X_1, \ldots,X_m)$ of the $m$-dimensional subspace $U$  one has
 \[
(\omega^A_{ij})=(<X_i,AX_j>)=
  \left\{
 \begin{array} {l}
 \geq 0  \quad  \mbox{  if $i$ is odd  and  } j=i+1,\\
     0  \quad \mbox{otherwise},
 \end{array}
 \right.
 \]
 for $i \leq j \leq k$.

A standard form   determines   some $\omega^A$-invariant subspaces $U_i^A$.\footnote{Using the same notation used in \cite {Vacpreprint} we call $\omega^A$-invariant subspaces the $T$-invariant subspaces of the endomorphism $T$  of $U$ represented by the same matrix of $\omega^A$.} Such subspaces are uniquely defined whereas  the standard bases of $\omega^A|_{U_i^A}$ are not: namely, given one of them, all others are obtained through an orthogonal transformations of $U_i^A$  commuting with the complex structure $A \in S(\mathcal{Q})$). Observe that all such bases have  the same orientation.

The following definitions  are specific of an Hermitian quaternionic vector space.
\begin{defi} \cite{BR1} \label{definitions of characteristic and Hermitain angle}
The \textbf{characteristic angle} $\theta \in [0, \frac{\pi}{2}]$ of a pair of vector $L,M$ is the angle of the 1-dimensional quaternionic subspace they generate  i.e. $\theta=\widehat{\mathcal{Q} L,\mathcal{Q} M}$ where $\mathcal{Q}L=L(L,IL,JL,KL)$, $\mathcal{Q} M=L(M,IM,JM,KM)$ and $(I,J,K)$ is any admissible basis. It is given by
\[\cos \theta (L,M)=\frac{[\mathcal{N}(L \cdot M)]^2}{mis L^4 \; mis M^4}= \frac{[<L,M>^2+ <L,IM>^2+<L,JM>^2+<L,KM>^2]^2}{{mis^4 L \; mis^4 M}}\]
The \textbf{Hermitian  angle} $\theta_h$ of a pair of vector $L,M$ is
\[\cos \theta_h (L,M)=\frac{|L \cdot M|}{mis L \; mis M}= \frac{\sqrt{<L,M>^2+ <L,IM>^2+<L,JM>^2+<L,KM>^2}}{mis L \; mis M}.\]
Then
\[\cos \theta (L,M)=\cos^4 \theta_h (L,M)\]
\end{defi}

We conclude this preliminary section recalling that the group $Sp(1)$ is  the group under multiplication of unitary quaternions.
It is a Lie group whose  Lie algebra $\mathfrak{sp}_1= Im \; \mathbb{H}$ is a quaternionic structure on $\mathbb{H}^n$. 

Let   $(V^{4n}, \mathcal{Q}, < , >)$  be an  Hermitian quaternionic  vector space. 
For any  quaternion $q \in Sp(1)$, let consider the unitary homothety in the $\mathbb{H}$-module $V \simeq \R^{4n}$. 
\[ R_q: X \mapsto Xq, \quad X \in V.\]
To these transformations belong for instance the automorphisms $I=R_{-i},J=R_{-j},K=R_{-k}$.
\begin{prop} \cite{BR2}
The unitary homotheties are rotations of $V^{4n}$ 
that preserve any quaternionic line. 
Moreover for any $X \in V$ the angle $\widehat{X,Xq}$ does not depend on $X$ and it is
\[
\cos \widehat{X,Xq}= Re(q)\]
\end{prop}

The action of $Sp(1)$ determines then an inclusion
\begin{equation}\label {inclusion Sp(1) in SO(4n)}
\begin{array}{lcll}
\lambda: &Sp(1)  &\hookrightarrow & SO(4n)\\
&q &\mapsto & R_q
\end{array}
\end{equation}

We define $Sp(n)$ to be the subgroup of $SO(4n)$ commuting with $\lambda(Sp(1))$ i.e. $Sp(n)$ is the centralizer of $\lambda Sp(1)$ in $SO(4n)$.

 For completeness we recall the following definition. Let  consider in $\mathbb{H}$-module $V$ the transformations   $T_{(A,q)}:X \mapsto AXq$ with $A \in Sp(n), \; q \in Sp(1), \; X \in V$. We denote by $Sp(n) \cdot Sp(1)$ the group of these transformations.

The group  $Sp(n) \cdot Sp(1)$ is the normalizer of $\lambda Sp(1)$ in $SO(4n)$ which is isomorphic to the quotient $Sp(n) \times_{\Z_2} Sp(1)$ where $\Z_2=\{1, -1\}$. Note that $Sp(1) \cdot Sp(1)$ is precisely $SO(4)$, whereas for $n \geq 2$  $ Sp(n) \cdot Sp(1)$ is a maximal Lie subgroup of
$SO(4n)$. Observe that $ Sp(n) \cdot Sp(1)$ is not a subgroup of $U(2n)$.
For a deeper understanding of the groups $Sp(n)$ and $Sp(n) \cdot Sp(1)$ one can refer among others to \cite{Sa} and {\cite{GR1}.

\section{$Sp(n)$-orbits of complex and $\Sigma$-complex subspaces in the real Grassmannians}

\subsection{Theorems about the $Sp(n)$-orbits of a generic subspace} \label{Theorems about the $Sp(n)$-orbits of a generic subspace}
Let $(V^{4n}, \mathcal{Q}, <, >)$) be  an Hermitian   quaternionic vector space.
Recalling the expression  of  the Hermitian product given in  (\ref{relation between scalar product and hermitian product of H^n}), in \cite{Vacpreprint} we find the following characterizations of the $Sp(n)$-orbits in the real Grassmannians $Gr^\R(m,4n)$.
\begin{teor} \cite{Vacpreprint} \label{transitivity of Sp(n) on subspaces with same Hermitian matrices}
Let $U$ and $W$ be a pair of subspaces of real dimension $m$ in the $\mathbb{H}$-module $V^{4n}$. Then there exists $A \in Sp(n)$ such that $AU=W$ iff
there exist 
  bases $\mathcal{ B}_U=(X_1, \ldots, X_m)$ and  $\mathcal {B}_W=(Y_1, \ldots, Y_m)$ of $U$ and $W$ respectively w.r.t. which for the Hermitian products one has  $(X_i \cdot X_j)= (Y_i \cdot Y_j), \; i=1, \ldots,m$ for one and hence any admissible basis of $\mathcal{Q}$.
\end{teor}

Let $U \subseteq V$ be a subspace. For any $A \in S(\mathcal{Q})$ we denoted by $\mathcal{B}^A(U)$ the set of the standard bases of $\omega^A|_U$ and by $\bm{\theta}^A(U)$ the principal angles between the pair $(U,AU)$. Moreover,
fixed an admissible basis $(I,J,K)$,   $\mathcal{B}(U)=\{B_I,B_J,B_K\}$ is the set made of   triples of  bases of $U$    with $B_I \in \mathcal{B}^I(U), B_J \in \mathcal{B}^J(U), B_K \in \mathcal{B}^K(U)$.

Necessary and sufficient conditions for the pair $U,W$ to belong to the same $Sp(n)$-orbit in the real Grassmannian  are also stated in
  \begin{teor} \cite{Vacpreprint} \label{main_theorem 1_ existance of canonical bases with same mutual position}
Let $(I,J,K)$ be an admissible basis of $\mathcal{Q}$. The subspaces  $U^m$ and $W^m$ of $V$ are in the same $Sp(n)$-orbit iff
    \begin{enumerate}
    \item   they share the same  $I,J,K$-principal angles i.e.
      \[\bm{\theta}^I(U)=\bm{\theta}^I(W),\quad \bm{\theta}^J(U)=\bm{\theta}^J(W),\quad \bm{\theta}^K(U)=\bm{\theta}^K(W)\] for one and hence any hypercomplex basis $(I,J,K)$ or equivalently      the singular values of the projectors \\ $Pr^{IU},Pr^{JU},Pr^{KU}$ equals those of $Pr^{IW},Pr^{JW},Pr^{KW}$ for one and hence any admissible basis $(I,J,K)$;
    \item there exist $(\{\tilde X_i \} , \{\tilde Y_i \} ,\{\tilde Z_i \}) \in \mathcal{B}(U)$   and  $(\{\tilde X'_i \},\{\tilde Y'_i \},\{\tilde Z'_i \}) \in \mathcal{B}(W)$ such that   $A= A', \qquad B= B'$  where
          \[A=(<\tilde X_i, \tilde Y_j>), \qquad A'=(<\tilde X_i', \tilde Y_j'>), \qquad B=(<\tilde X_i, \tilde Z_j>),\qquad  B'=(<\tilde X_i', \tilde Z_j'>).\]
     \end{enumerate}
\end{teor}

The determination of the  principal angles between a pair of subspaces $S,T$ is a well know problem solved by the singular value decomposition of the orthogonal projector  of $S$ onto $T$. Here, for a chosen $U \subset V$ we consider  the pairs $(U,AU),  \; A \in S(\mathcal{Q})$  and denote by  $Pr^{AU}$ the orthogonal projector of $U$ onto $AU$. In this case the singular values of  $Pr^{AU}$ are always degenerate which implies that they have non-unique singular vectors. In terms of the principal vectors of the pair $(U,AU)$ we can equivalently say that the  principal vectors are never uniquely defined not even if $U$ is  2-dimensional. We do not consider the ambiguity in sign of the principal vectors which is always existing also in dimension one.

Another way to obtain the principal angles and the associated principal vectors between the pair of subspaces $(U,AU)$,
for any $A \in S(\mathcal{Q})$,  is through the  standard decompositions of the restriction to $U$ of the $A$-K\"{a}hler skew-symmetric form  $\omega^A: (X,Y) \mapsto <X,AY>,  \; X,Y \in U$.

Recalling that  a \textit{standard basis} is  an orthonormal basis w.r.t. which $\omega^A|_U$ assumes standard form, in the following we will consider it ordered according to non increasing value of the non-negative entries.
Denoting by $\mathcal{B}^A(U)$ the set of all such bases, one has that any $B \in \mathcal{B}^A(U)$ consists of the  principal vectors  of the pair $(U,AU)$ (ordered as just explained).

The problem to determine the $Sp(n)$-orbits turns therefore into the one of determining the existence of such triple of  bases. In  \cite{Vacpreprint} we set up a procedure  to determine a triple of  \textit{canonical bases} w.r.t. which one  computes the matrices $A$ and $B$. Namely, fixed an admissible basis $(I,J,K)$, a triple of canonical bases of a subspace $U$ is constituted by
a triple of standard bases of $\omega^I,\omega^J,\omega^K$ that either  are  uniquely determined by the procedure aforementioned or, if this is not the case, nevertheless the associated matrices $A,B$ are unique. We call them \textit{canonical matrices} and denote by  $C_{IJ}$ and $C_{IK}$.
If the pair $U,W$ belong the same $Sp(n)$-orbit,  the action of the group maps the canonical basis of $U$ onto the ones of $W$. Therefore the determination of such standard bases let us restrict the statement of Theorem (\ref{main_theorem 1_ existance of canonical bases with same mutual position}) to the canonical matrices leading to the

  \begin{teor} \cite{Vacpreprint} \label{main_theorem 1_ with respect canonical bases}
Let $(I,J,K)$ be  an admissible  basis of $\mathcal{Q}$. The subspaces  $U$ and $W$ of $V$ are in the same $Sp(n)$-orbit  iff $Inv(U)=Inv(W)$ where
\[Inv(U)=\{\bm{\theta}^I(U), \, \bm{\theta}^J(U), \,   \bm{\theta}^K(U), C_{IJ}, \; C_{IK}\}.\]

\end{teor}
It is straightforward to verify that if $Inv(U)=Inv(W)$ w.r.t. the admissible basis $(I,J,K)$ then $Inv(U)=Inv(W)$ for any admissible  basis.


\subsection{The 2-dimensional complex subspace }

Let $(V^{4n}, \mathcal{Q}, < , >)$ be an Hermitian quaternionic  vector space. To study  the $Sp(n)$-orbits of a $2m$-dimensional complex and $\Sigma$-complex subspace of $V$ in the real Grassmannian $Gr^\R(2m,4n)$ we will need the theory of the isoclinic subspaces developed in \cite{Vac_isoclinic}. In particular here we use some of the results obtained regarding the 4-dimensional case as we will see that complex and $\Sigma$-complex subspaces admit an orthogonal  decomposition into 4-dimensional  isoclinic addends (plus a totally complex 2-dimensional subspace if $m$ is odd). Furthermore such addends will be complex by some compatible complex structure and  Hermitian orthogonal in pairs.

In the following then we recall  some results about isoclinic subspaces  referring for a more extensive analysis to \cite{Vac_isoclinic}.

Let initiate our study by considering a 2-dimensional $I$-complex subspace.
 By the skew-symmetry of the $A$-K\"{a}hler form for any $A \in  S(\mathcal{Q})$, any  2-dimensional subspace of $U$ is isoclinic with $AU$.  Therefore  as a set one has that $G_\R(2,4n)= \mathcal{IC}^2$.

 In particular all 2-dimensional complex subspaces are totally complex. The invariants characterizing the $Sp(n)$-orbits in the Grassmannian of 2-planes are determined in  \cite{Vac} where we also studied the analogue problem for the group  $Sp(n) \cdot Sp(1)$.

 Fixed an admissible basis $(I,J,K)$, let consider an oriented 2-dimensional subspace $U \subset V$ generated by the  oriented basis  $(M,L)$. In  \cite{Vac} we introduced the  purely imaginary quaternion
\begin{equation} \label{definition of imaginary measure}
 \mathcal{IM}(U)=
 \frac{Im (L \cdot M)}{mis(L \wedge M)}, \quad  L,M \in U.
 \end{equation}
 We showed  that it is  an intrinsic property of  an oriented 2-plane \mbox{$U \subset (V^{4n}, \mathcal{Q}, <,>)$}  i.e. it does not depend neither on  the chosen generators $L,M$ nor on the admissible  basis $\mathcal{H}$ of $\mathcal{Q}$.  Moreover $Sp(n)$ preserves $\mathcal{IM}(U)$.
In particular, if the pair  $L,M$ is an orthonormal oriented  basis of $U$, then $\mathcal{IM}(U)=L \cdot M$.
   We called it \textbf{Imaginary measure} of the 2-plane $U$.
     Disregarding the orientation of  $U$  and being  $(L,M)$  some orthonormal basis,  it is $\mathcal{IM}(U)= \{\pm \, L \cdot M\}$ i.e.  it is the set made of a pair of conjugated pure imaginary quaternions.
  We  proved   that
   \begin{teor} \cite{Vac}  \label{Sp(n)-orbit of a 2-plane}
   The imaginary measure $\mathcal{IM}(U)$ represents  the full system of invariants for the $Sp(n)$-orbits in the real Grassmannian of 2-planes $G_\R(2,4n)$ as well as in $G_\R^+(2,4n)$ (the Grassmannian  of the oriented  2-planes) i.e. a pair of 2-planes $(U,W)$ of $(V^{4n}, \mathcal{Q}, <,>)$ are in the same $Sp(n)$-orbit iff $\mathcal{IM}(U)=\mathcal{IM}(W)$.
 \end{teor}

Let   consider   a triple of  standard bases $(X_1,X_2), (X_1,Y_2),(X_1,Z_2)$ with a common leading vector $X_1$ of the non oriented 2-plane $U$. 
 By definition one has that  $\cos \theta^I=<X_1,IX_2>,\cos \theta^J=<X_1,JY_2>,\cos \theta^K=<X_1,KZ_2>$ are non negative, and computed $\xi=<X_2,Y_2>,\chi=<X_2,Z_2>, \eta=<Y_2,Z_2>$, where $(\xi,\chi,\eta) \in \{-1,1\}$ one has that
the matrices $C_{IJ}$ and  $C_{IK}$ w.r.t. the standard bases $(X_1,X_2),(X_1,Y_2)(X_1,Z_2)$  are given by 
\begin{equation} \label{canonical matrices of a 2-dimensional subspace}
C_{IJ}:
 \left(   \begin{array}{cc}
1 & 0 \\
0 & \xi
\end{array}      \right) \qquad
C_{IK}:
 \left(   \begin{array}{cc}
1 & 0 \\
0 & \chi
\end{array}
\right)
\end{equation}
It is straightforward to verify that  the pair $(\xi,\chi)$ is an invariant of $U$. Then any triple of standard bases of $\omega^I,\omega^J,\omega^k$ with a common leading vector are canonical bases of $U$ whose canonical matrices are given in (\ref{canonical matrices of a 2-dimensional subspace}).
 Therefore, according to Theorem (\ref{main_theorem 1_ existance of canonical bases with same mutual position}), the pair $(\xi,\chi)$ together with the triple $(\theta^I,\theta^J,\theta^K)$,
 determines  the $Sp(n)$-orbits of the (non oriented) 2-plane $U$.

This is accordance with the Theorem (\ref{Sp(n)-orbit of a 2-plane}). In fact,   If $U$ has a triple orthogonality then  clearly $\mathcal{IM}(U)=0$. In this case any orthonormal basis is at the same time a  standard basis of $\omega^I,\omega^J,\omega^K$ which implies $\xi=\chi=1$.   Else,  suppose without lack of generality that $\cos \theta^I \neq 0$ and  let
$(X_1,X_2)$ be an $\omega^I$-standard basis. Then
\[\mathcal{IM}(U)= X_1 \cdot X_2= \pm( \cos \theta^I i + \xi \cos \theta^J j + \chi \cos  \theta^K k).\]

 Given a pair  of 2-planes  $U,W$ with $\mathcal{IM}(U)=\mathcal{IM}(W)$,  according to Theorem (\ref{Sp(n)-orbit of a 2-plane}), they are in the same orbit. Since   they share the same  pair $(\xi,\eta)$ and the same triple $(\theta^I,\theta^J,\theta^K)$,   they are in the same $Sp(n)$-orbit also according to
Theorem (\ref{main_theorem 1_ existance of canonical bases with same mutual position}). Viceversa  if they share the same  pair $(\xi,\eta)$ and the same triple $(\theta^I,\theta^J,\theta^K)$  which implies that they belong to the same $Sp(n)$-orbit according to  Theorem (\ref{main_theorem 1_ existance of canonical bases with same mutual position}) then clearly  $\mathcal{IM}(U)=\mathcal{IM}(W)$.

A 2-dimensional $I$-complex subspace is totally complex being clearly  $U \perp JU=KU$. Then for any admissible basis $(I,J,K)$ one has $(\theta^I,\theta^J,\theta^K)=(1,0,0)$. Furthermore for $X \in U$ the $\omega^I$-standard basis  $(X,-IX)$ can be considered a standard basis of $\omega^{K'}$ centered on $X$ for any $K' \in I^\perp$. Then we assume $\xi=\chi=1$ and w.r.t. any standard and canonical $\omega^I$-basis, for the canonical matrices one has $C_{IJ}= C_{IK}=Id$. Then we can conclude affirming
\begin{teor} \label{Sp(n) orbita di un 2 piano complesso}
All and only the $I$-complex 2-dimensional subspaces of  $(V^{4n}, \mathcal{Q}, < , >)$  belong to one $Sp(n)$-orbit  in $Gr^\R(2,4n)$ or equivalently $Sp(n)$ is transitive on the set of the 2-dimensional $I$-complex subspaces of $V$.
\end{teor}

We obtain the same result applying the Theorem (\ref{Sp(n)-orbit of a 2-plane}) since  if $U$ is an  $I$-complex  2-plane then  $\mathcal{IM}(U)= \pm i$.

Furthermore, w.r.t. any $\omega^I$-standard basis, the matrix of the Hermitian product  is given by
\begin{equation} \label{Hermitian matrix of a 2 dimensional complex subspace}
H_\mathcal{B}(U)= \left(   \begin{array}{cc}
0  &  i \\
\\
-i &  0
\end{array}    \right).
\end{equation}

  Then again the same results follows from Theorem (\ref{transitivity of Sp(n) on subspaces with same Hermitian matrices}). 

\vskip .4cm

For completeness  we report the corresponding result appearing in   \cite{Vac} for the group $Sp(n) \cdot Sp(1)$.  There, we first
extended to the Hermitian quaternionic vector space $V$  some notions and results of a vector space endowed with a complex structure (see \cite{Riz2}).  In \cite{BR1},  considering a 2-plane $U\subset V$ spanned by  the pair $(L,M)$, it has been introduced
\begin{equation} \label{characteristic deviation}
\Delta(U)=\mathcal{N}(\mathcal{IM}(U))=\frac{\mathcal{N} [Im(L \cdot M)]}{mis^2(L \wedge M)}=\frac{<L,IM>^2+<L,JM>^2+<L,KM>^2}{mis^2(L \wedge M)}.
\end{equation}
In particular, in case the basis $L,M$ is orthonormal, $\Delta(U)= \mathcal{N} (L \cdot M)$.
It  is a real number  belonging to the close interval $[0,1]$ and equals 1 iff  $\dim U^\mathbb{H}=1$.
It has been proved that the quantity $\Delta(U)$  is an intrinsic property of a 2-plane (see (\ref{sum of square cosines of I,J,K-kaehòer angles of a 2-plane}))  preserved by the action of the group $Sp(n) \cdot Sp(1)$ on $V$.

We called the angle $\delta(U) \in [0,\pi/2]$ such that  $\cos^2 \delta(U)= \Delta(U)$
 the  \textbf{characteristic deviation of the real 2-plane} $U \subset V$. Moreover
\[\Delta (U)= cos^2 \delta(U)= \cos (\widehat{U,IU})  + \cos (\widehat{U,JU}) + \cos (\widehat{U,KU})\]
where $\cos(\widehat{U,IU})$ (resp. $\cos(\widehat{U,JU})$, $\cos(\widehat{U,KU})$) denotes the cosine of the angle between the pairs  of 2-planes $(U,IU)$ (resp. $(U,JU)$, $(U,KU)$).

In Proposition (\ref{invariance of the sum of the 3 angles (U,IU),(U,JU),(U,KU)}) we showed that such  quantity does not depend on the admissible basis $(I,J,K)$ of $\mathcal{Q}$.

 Generalizing the definition of the characteristic deviation given for a 2-plane,  in  (\cite{Vac}) we find the following
\begin{defi} \label{definition of characteristic deviation of any subspace}
Let $(X_1, \ldots, X_m)$ be an orthonormal basis of an $m$-dimensional subspace $U$. Denote by  $U_{rs}=L(X_r, X_s), \, r \neq s=1, \ldots, m$. We call the quantity
\begin{equation}\label{characteristic deviation of a subspace}
\Delta(U)=\left( \begin{array}{c} m \\ 2 \end{array} \right)^{-1} \sum_{r<s}  \Delta(U_{rs}) \end{equation}
the \textbf{characteristic deviation of the subspace $U^m$}.
\end{defi}

There we proved that the characteristic deviation   is an intrinsic property of a subspace $U \subset V$ i.e. depends neither on the  admissible basis  of  $\mathcal{Q}$  (it is a consequence of Proposition (\ref{invariance of the sum of the 3 angles (U,IU),(U,JU),(U,KU)})) nor on the chosen orthonormal basis of $U$ which determines the 2-planes $U_{rs}$ in   (\ref{characteristic deviation of a subspace}).  We proved the following

\begin{teor} \cite{Vac}
The characteristic deviation $\delta$ determines completely the orbit of the  2-plane $U\subset \mathbb{H}^n$ in the real Grassmannian $G_\R(2,4n)$ under the action of $Sp(n) \cdot Sp(1)$
\end{teor}

Since for of a 2-dimensional complex subspace  $U$ one has $\Delta(U)=1$ we can state the
\begin{coro}
The group $Sp(n) \cdot Sp(1)$ acts transitively on the set of 2-dimensional complex subspaces.
\end{coro}

\subsection{Some results for isoclinic 4-dimensional  subspaces}
Let $I \in S(\mathcal{Q})$. We now consider the case of a pure $I$-complex subspace $(U,I)$ of the Hermitian quaternionic vector space $(V^{4n}, \mathcal{Q},<,>)$.  In point (\ref{for any K in I perp  KU=JU is I-complex}) of the Claim (\ref{properties of some quternionic,complex and real subspaces}) we  observed that, if $K,K' \in I^\perp \cap  S(\mathcal{Q})$  then $KU=K'U$. Furthermore the subspace $KU$ is $I$-complex. We will deal first with a 4-dimensional pure complex  subspace since, we will show later, it is the fundamental brick to   determine the $Sp(n)$-orbit of The complex and $\Sigma$-complex subspaces.

As we will soon show, a 4-dimensional complex subspace $U$ is isoclinic i.e.  $U \in \mathcal{IC}^4$.  The study of isoclinic subspaces is carried on in \cite{Vac_isoclinic}. Here we briefly recall the results we need in this paper 
referring to  \cite{Vac_isoclinic} both  for proofs and a deeper analysis.

\begin{prop} \cite{Vac_isoclinic} \label{general form for the matrix of projector omega^A of an isoclinic 4 dimensional subspace}
Let $A \in S(\mathcal{Q})$. The pair $(U,AU)$ of 4-dimensional subspaces is isoclinic iff the matrix of $\omega^A$ w.r.t. the orthonormal basis $(X_1,X_2,X_3,X_4)$  has the form
\begin{equation}\label{matrix of projector omega^A of an isoclinic 4 dimensional subspace}
\omega^A:
 \left(   \begin{array}{cccc}
0 & a & b & c\\
-a & 0 & \pm c & \mp b\\
-b & \mp  c & 0 & \pm a\\
-c & \pm b & \mp a & 0
\end{array}      \right).
\end{equation}
\end{prop}
It is a matrix with orthogonal rows and columns  
whose square norms  evidently  equal the square cosine of  the angle of isoclinicity  $\theta^A$ between the pair $(U,AU)$ i.e. 
\[
\cos \theta^A= \sqrt{a^2 + b^2 + c^2}.
\]

Then we proved the following result which is valid only in dimension  4 (besides obviously  in dimension 2 being all 2-planes  isoclinic).
\begin{prop} \cite{Vac_isoclinic} \label{isoclinicity w.r.t. one hypercomplex basis implies isoclinicity w.r.t. any compatible complex structure}
Let $U$ be  a 4 dimensional subspace and $(I,J,K)$ an admissible basis. Suppose the pairs $(U,IU)$, $(U,JU)$, $(U,KU)$ are isoclinic and $\theta^I, \theta^J, \theta^K$ the respective angles of isoclinicity. Then for any  $A= \alpha_1 I + \alpha_2 J + \alpha_3 K \in S\mathcal(Q)$  the pair $(U,AU)$ is isoclinic and therefore $U \in \mathcal{IC}^4$.
The angle of isoclinicity $\theta^A$ between the pair $(U,AU)$  is given by
\begin{equation} \label{angle of isoclinicity of a 4 dimensional subspace isoclinic w.r.t. an hypercomplex basis 1}
\cos^2 \theta^A =-\frac{1}{4} Tr [(\alpha_1 \omega^I + \alpha_2 \omega^J + \alpha_3 \omega^K)^2]= -\frac{1}{4} Tr [(\omega^A)^2]
\end{equation}
\end{prop}

We recall that a 4-dimensional complex subspace is never orthogonal unless it is totally complex  (in which case it has a double orthogonality).
 Suppose that $U$ is not an orthogonal subspace (w.r.t. $(I,J,K)$)  and let
   \begin{equation} \label{$X_2,Y_2,Z_2$}
  X_2=\frac{I^{-1} Pr^{IU}X_1}{\cos \theta^I}, \; Y_2=\frac{J^{-1} Pr^{JU}X_1}{\cos \theta^J}, \; Z_2=\frac{K^{-1} Pr^{KU}X_1}{\cos \theta^K}
  \end{equation}
 be the unitary vectors such that $(X_1,X_2)$, $(X_1,Y_2)$, $(X_1,Z_2)$ are  (orthonormal) standard bases of the  standard 2-planes $U^I,U^J,U^K$ of $\omega^I,\omega^J,\omega^K$  respectively they generate. The quantities  $<X_1,IX_2>, <X_1,JY_2>,<X_1,KZ_2>$   are the (non negative) cosines of the principal angles of the pairs $(U^I,IU^I),(U^J,JU^J),(U^K,KU^K)$ or equivalently the absolute value of the cosine of the $I,J,K$-K\"{a}hler angles of the 2-planes $U^I,U^J,U^K$ respectively.
 Let $(X_1,X_2,X_3,X_4)$ (resp. $(X_1,Y_2,Y_3,Y_4)$, resp. $(X_1,Z_2,Z_3,Z_4)$) be a standard basis  of the  forms  $\omega^I$ (resp. $\omega^J$, resp. $\omega^K$ ) with the common leading vector $X_1$. Let moreover  denote by
 \[\xi=<X_2,Y_2>,  \quad \chi=<X_2,Z_2>, \quad \eta=<Y_2,Z_2>\]
   where $\xi, \chi,\eta \in [-1,1]$.

\begin{prop} \cite{Vac_isoclinic} \label{invariance of <X_2,Y_2>}
  The cosines  $\xi=<X_2,Y_2>,  \; \chi=<X_2,Z_2>, \; \eta=<Y_2,Z_2>$ are invariants of $U$.
\end{prop}
They are not an intrinsic properties of $U$  i.e. they depend on the  chosen admissible basis (see \cite{Vac_isoclinic}).

\begin{defi}
Let $(I,J,K)$ be an admissible basis. The subspace  $U \in \mathcal{IC}^4$ with angles $(\theta^I,\theta^I,\theta^K)$ is said to be a \textit{2-planes decomposable} subspace  (or simply 2-planes decomposable) if  it admits an orthogonal decomposition into  a pair of 2-planes both isoclinic with their $I,J,K$-images with angles $(\theta^I,\theta^I,\theta^K) $ respectively.
\end{defi}
In a 2-plane decomposable  subspace  the values of $\xi,\chi,\eta=\xi \cdot \chi$ are all clearly equal to $\pm 1$. In case of a 4 dimensional complex subspace $U$ this happens iff $U$ is totally complex.
It is straightforward to verify that if $U \in \mathcal{IC}^4$ is 2-plane decomposable w.r.t $(I,J,K)$ then it is 2-planes decomposable w.r.t. any admissible basis.

By using the invariants $(\xi,\chi,\eta)$, in \cite{Vac_isoclinic} we determined  the following equivalent expression for the angle of isoclinicity
 \begin{equation}\label{equivalent form for the angle of isoclinicity of a 4 dimensional subspace isoclinic w.r.t. an hypercomplex basis}
\begin{array}{l}
\cos^2 \theta^A  = \alpha_1^2 \cos^2 \theta^I + \alpha_2^2\cos^2 \theta^J + \alpha_3^2\cos^2 \theta^K +
  2  \xi \alpha_1 \alpha_2 \cos \theta^I \cos \theta^J + 2 \chi \alpha_1 \alpha_3 \cos \theta^I  \cos \theta^K+ 2 \eta \alpha_2
\alpha_3 \cos \theta^J \cos \theta^K.
\end{array}
\end{equation}


\begin{prop} \cite{Vac_isoclinic} \label{$S$ sum of square cosined of the angles of isoclinicity (U,IU),(U,JU),(U,KU) of isoclinic subspaces}
Let $(I,J,K)$ be an admissible hypercomplex basis and  $U \in \mathcal{IC}^4$
 with angles  of isoclinicity equal  respectively to $(\theta^I,\theta^J,\theta^K)$. Then $S= \cos^2 \theta^I + \cos^2 \theta^J + \cos^2 \theta^K$  is an intrinsic property of $U$  not depending on the admissible basis.
\end{prop}

 The previous property is more general in the sense that it is valid  for all subspaces $U \in \mathcal{IC}$ regardless their  dimension.

 Dealing with a complex subspace $U$ we have two possibilities: either  $U$ is not totally complex in which case none among $\xi,\chi,\eta$ equals $\pm 1$ or $U$ has a double orthogonality,  is 2-planes decomposable and  in particular $\xi=\chi= \eta=1$.

In order to determine the canonical bases in this two cases, let consider first the case that none among $\xi,\chi,\eta$ equals $\pm 1$.
    Let $X_1 \in U$ unitary  and   $(X_2,Y_2,Z_2)$ as in  (\ref{$X_2,Y_2,Z_2$}).
The subspaces $L(X_1,X_2),L(X_1,Y_2),L(X_1,Z_2)$ are respectively standard subspaces of $\omega^I, \omega^J,\omega^K$ restricted to $U$.
Consider  $L(X_2,Y_2)$ and the vectors $X_4,Y_4$ of such 2-plane such that $(X_2,X_4$) and $(Y_2,Y_4)$ are   a pair of orthonormal basis consistently oriented   with $(X_2,Y_2)$ (then  $<X_2,Y_4> <0$).

Let then $X_3=-\frac{I^{-1} Pr^{IU} X_4}{\cos \theta^I}$ be the  unique vector such that $<X_3,IX_4>= \cos \theta^I$ i.e. $L(X_3,X_4)$ is an $\omega^I$-standard 2-plane and analogously  $Y_3=-\frac{J^{-1}Pr^{JU}Y_4}{\cos \theta^J}$   the unique vector such that $<Y_3,JY_4>= \cos \theta^J$. Clearly the vectors $X_3$ and $Y_3$ belong to $U$.

Analogously  we consider   $L(X_2,Z_2)$ and the vectors $\tilde X_4,Z_4$ of such 2-plane such that $(X_2,\tilde X_4$) and $(Z_2,Z_4)$ are   a pair of orthonormal basis consistently oriented  with the pair $(X_2,Z_2)$ (then  $<X_2,Z_4> <0$). Again, let $\tilde X_3=-\frac{I^{-1} Pr^{IU} \tilde X_4}{\cos \theta^I}$ be the unique vector such that $<\tilde X_3,I\tilde X_4>= \cos \theta^I$ i.e. $L(\tilde X_3,\tilde X_4)$ is an  $\omega^I$ standard 2-plane and $Z_3=-\frac{K^{-1}Pr^{KU}Z_4}{\cos \theta^K}$  the unique vector such that $<Z_3,K Z_4>= \cos \theta^K$ i.e.  $L (Z_3,Z_4)$ is an  $\omega^K$ standard 2-plane. The vectors $\tilde X_3$ and $Z_3$ belong to $U$. Proceeding in the same way considering the oriented 2-plane $L(Y_2,Z_2)$ we determine the pair $(\tilde Y_4, \tilde Z_4)$ and consequently $(\tilde Y_3,\tilde Z_3)$.
With the above choices in \cite{Vac_isoclinic} we proved the following

\begin{prop} \cite{Vac_isoclinic} \label{equality of thirs element $X_3=Y_3$ of the chains} $ $\\
$X_3= \frac{I Pr^{IU} X_4}{\cos \theta^I}= \frac{JPr^{JU}Y_4}{\cos \theta^J}=Y_3, \qquad
\tilde X_3= \frac{IPr^{IU} \tilde X_4}{\cos \theta^I}= \frac{KPr^{KU}Z_4}{\cos \theta^K}=Z_3, \qquad \tilde Y_3= \frac{JPr^{IU} \tilde Y_4}{\cos \theta^J}= \frac{KPr^{KU}\tilde Z_4}{\cos \theta^K}=\tilde Z_3$.
\end{prop}

\begin{defi} \label{associated chains} Let $U \in \mathcal{IC}^4$, $(I,J,K)$  be an admissible basis and $(\theta^I, \theta^J, \theta^K)$ the respective angles of isoclinicity. In case none among $\xi,\chi,\eta$ is equal to $\pm 1$  (in particular if $U$ is  neither orthogonal nor 2-planes decomposable), for any unitary $X_1 \in U$, that we call leading vector, we define  the following  standard bases of $\omega^I|_U$ and $\omega^J|_U$ respectively
\[
\begin{array}{l}
\{X_i\}=\{X_1, X_2= \frac{I^{-1} Pr^{IU}X_1}{\cos \theta^I}, X_3=-\frac{I^{-1} Pr^{IU}X_4}{\cos \theta^I}, X_4=  \frac{Y_2- \xi X_2}{\sqrt{1-\xi^2}}\},\\
\{Y_i\}=\{X_1, Y_2= \frac{(J^{-1} Pr^{JU}X_1)}{\cos \theta^J}, Y_3=X_3=-\frac{I^{-1} Pr^{JU}Y_4}{\cos \theta^J}, Y_4  =  \frac{-X_2+ \xi Y_2}{\sqrt{1-\xi^2}} \}
\end{array}
\]
 the  $\omega^I$ and $\omega^J$-\textbf{chains} of $U$  centered on $X_1$,  and the following  standard bases of $\omega^I|_U$ and $\omega^K|_U$ respectively
\[
\begin{array}{l}
\{\tilde X_i\}=\{X_1, X_2=   \frac{I^{-1} Pr^{IU}X_1}{\cos \theta^I}, \tilde X_3=-\frac{I^{-1} Pr^{IU} \tilde X_4}{\cos \theta^I}, \tilde X_4=  \frac{Z_2- \chi X_2}{\sqrt{1-\chi^2}}\},\\
 \{Z_i\}=\{X_1, Z_2=   \frac{K{-1} Pr^{KU}X_1}{\cos \theta^K}, Z_3=\tilde X_3=-\frac{K^{-1} Pr^{KU}  Z_4}{\cos \theta^K}, Z_4  =   \frac{-X_2+ \chi Z_2}{\sqrt{1-\chi^2}}\}
\end{array}
\]
the  $\omega^I$ and  $\omega^K$-\textbf{chains} of $U$ centered on  $X_1$ and the following  standard bases of $\omega^J|_U$ and $\omega^K|_U$ respectively.
\[
\begin{array}{l}
\{\tilde Y_i\}=\{X_1, Y_2=   \frac{J^{-1} Pr^{JU}X_1}{\cos \theta^J}, \tilde Y_3=-\frac{J^{-1} Pr^{JU} \tilde Y_4}{\cos \theta^J}, \tilde Y_4=\frac{Z_2- \eta Y_2}{\sqrt{1-\eta^2}}\},\\
 \{\tilde Z_i\}=\{X_1, Z_2=   \frac{K^{-1} Pr^{KU}X_1}{\cos \theta^K}, \tilde Z_3=\tilde Y_3=-\frac{K^{-1} Pr^{KU} \tilde  Z_4}{\cos \theta^K}, \tilde Z_4  =   \frac{-Y_2+ \eta Z_2}{\sqrt{1-\eta^2}}\}
\end{array}
\]
the  $\omega^J$ and  $\omega^K$-\textbf{chains} of $U$ centered on  $X_1$.
We denote by  $\Sigma (X_1)$  the set of the six chains with leading vector $X_1$.
\end{defi}

Clearly $\Sigma (X_1)$ is  uniquely determined by the leading vector $X_1$.

 In case a  pair among $(\xi,\chi,\eta)$  and hence all three pairs are equal to  $\pm 1$,  (namely either $\eta=\xi= \chi=1 $ or  two of them are equal to  -1 and the other  to 1),  we give the following
 \begin{defi} \label{The chains if all the 3 invariants is one or minus one}
In case  $U$ is a 2-planes decomposable subspace i.e. $\xi,\chi,\eta$ are all equal to  $\pm 1$ we define the following chains:
   \[\{X_i \}= \{\tilde X_i \}, \; \{Y_i \}=(X_1, \xi X_2,X_3,\xi X_4)= \{\tilde Y_i \}, \; \{Z_i \}=(X_1, \chi X_2,X_3,\chi X_4)=(X_1,\eta Y_2,X_3,\eta Y_4)= \{\tilde Z_i \}.\]
   In particular  if  $U$ has a double  orthogonality (which happens in particular if $U$ is totally complex)    or a triple orthogonality (iff $U$ is a totally real subspace) one has   $\{X_i \}=\{\tilde X_i \}=\{Y_i \}= \{\tilde Y_i \}=\{Z_i \}= \{\tilde Z_i \}$.
\end{defi}

Clearly $L(X_3,X_4)=L(\tilde X_3,\tilde X_4)$. The bases $(X_3,X_4)$ and $(\tilde X_3,\tilde X_4)$, being $\omega^I$-standard bases, are consistently oriented. Let
 \[
 C: \left(   \begin{array}{cc}
 <X_3,\tilde X_3> & <X_3,\tilde X_4>\\
<X_4,\tilde X_3> & <X_4,\tilde X_4>
\end{array}      \right)
 = \left(   \begin{array}{cc}
 \Gamma & -\Delta\\
\Delta & \Gamma
\end{array}
\right)
\]
the orthogonal matrix of the change of basis. 
The orthogonal matrices $C_{IJ}=(<X_i,Y_j>)$ and  $C_{IK}=(<X_i,Z_j>)$ of the relative position of the basis $\{X_i\}=(X_1,X_2,X_3,X_4)$,  $\{Y_i\}=(X_1,Y_2,X_3,Y_4)$ and $\{Z_i\}=(X_1,Z_2,\tilde X_3,Z_4)$   are given by
 \begin{equation} \label{canonical matrices $C_{IJ}$ $C_{IK}$ and of 4-dimensional isoclinic subspace w.r.t. the associated chains}
C_{IJ}=\left(
\begin{array} {cccc}
1 & 0 & 0 & 0\\
0 & \xi & 0 & - \sqrt{1-\xi^2} \\
0 & 0 & 1 & 0 \\
0 &  \sqrt{1-\xi^2} & 0 & \xi
\end{array}
\right), \qquad
  C_{IK}=\left(
  \begin{array} {cccc}
1 & 0 & 0 & 0\\
0 & \chi & 0 & - \sqrt{1-\chi^2}\\
0 & -\Delta \sqrt{1-\chi^2} &\Gamma & -\chi \Delta \\
0 &  \Gamma  \sqrt{1-\chi^2} & \Delta & \chi \Gamma.
\end{array}
\right).
\end{equation}

To determine $\Gamma=<X_3,\tilde X_3>$,  being $<Y_2,Z_2>= <Y_2,X_2><Z_2,X_2>+  <Y_2,X_4><Z_2,X_4>$,
 we get $\eta= \xi \chi  + \sqrt{1-\xi^2}\sqrt{1-\chi^2} \; \Gamma$. From the above expression, in case neither $\xi$ nor $\chi$ equal 1, we get:
  \begin{equation} \label{cos phi}
      \Gamma= \frac{\eta-\xi \chi}{\sqrt{1-\xi^2} \; \sqrt{1-\chi^2} } \in [-1,1].
    \end{equation}

\begin{prop} \cite{Vac_isoclinic} \label{Gamma is an invariant of an isoclinic 4 dimensional subspace 1}
If none among $\xi,\chi,\eta$ is equal to $\pm 1$
 the value of   $\Gamma \in [-1,1]$ is  given in (\ref{cos phi}). 
 If instead at least one among $\xi,\chi,\eta$ is equal to $\pm 1$ then $\Gamma=1$.  In particular this happens if  $U$ is orthogonal or is a 2-planes decomposable subspace.
  In all cases,  the pair $(\Gamma,\Delta)$ is an invariant of $U$.   
\end{prop}

\begin{prop} \cite{Vac_isoclinic} \label{invariance of $C_{IJ}$ and $C_{IK}$ w.r.t. any leading vector}
The matrices $C_{IJ}$ and $C_{IK}$  given in  (\ref{canonical matrices $C_{IJ}$ $C_{IK}$ and of 4-dimensional isoclinic subspace w.r.t. the associated chains})
w.r.t. the   chains $\{X_i \}, \{Y_i \}$ and $\{X_i \}, \{Z_i \}$ centered on a common leading vector  are invariant of $U$.
\end{prop}

Since if   $U$ is orthogonal one has $(\Gamma,\Delta)=(1,0)$,  we can deduce the following
\begin{coro} \label{some of the pair $(U,IU), (U,JU),(U,KU)$ is strictly orthogonal}
In case of double or triple orthogonality  one has  $C_{IJ}=C_{IK}= Id$.

\end{coro}


Following the definition  given in \cite{Vacpreprint}, we give the
\begin{defi} \label{definiton of canonical bases and matrices of Is(4)}
 Let $U \in \mathcal{IC}^4$. Fixed  an admissible basis $(I,J,K)$, for any leading vector $X_1$, we call the chains $\{X_i\},\{Y_i\},\{Z_i\}$ (resp. the matrices $C_{IJ}$ and $C_{IK}$)  determined above \textbf{canonical bases}  (resp. \textbf{canonical matrices}) of the subspaces  $U \in \mathcal{IC}^4$.
\end{defi}
 Clearly for any leading vector we have  a different set of canonical bases. As explained  beforehand,   we denote them "canonical" since, by the invariance of  $(\xi,\chi,\eta,\Delta)$,  the matrices $C_{IJ}$ and $C_{IK}$ are invariants of $U \in \mathcal{IC}^4$ having the unique forms  given in (\ref{canonical matrices $C_{IJ}$ $C_{IK}$ and of 4-dimensional isoclinic subspace w.r.t. the associated chains})
 regardless the leading vector $X_1$.  We summarize the results obtained in the following
\begin{prop} \cite{Vac_isoclinic} \label{canonical matrices of any  subspace of ${IC}^(4)$}
Fixed an admissible basis $(I,J,K)$, to any $U \in \mathcal{IC}^4$   we can associate  the orthogonal canonical matrices $C_{IJ}$ and $C_{IK}$   given in
(\ref{canonical matrices $C_{IJ}$ $C_{IK}$ and of 4-dimensional isoclinic subspace w.r.t. the associated chains})
 representing the mutual position of the canonical (standard) bases $\{X_i\},\{Y_i\},\{Z_i\}$ of $\omega^I,\omega^J,\omega^K$. Such matrices depend  on
the triple of invariants $(\xi,\chi,\eta)$ and on the sign of $\Delta=<X_4,\tilde X_3>= \pm \sqrt{1-\Gamma^2}$  where $\Gamma= \Gamma(\xi,\chi,\eta)$ is  given in (\ref{cos phi}) if  none among $\xi,\chi,\eta$ is equal to $\pm 1$  else $\Gamma=1$. The second case happens in particular    if  $U$ is orthogonal or a 2-planes decomposable subspace. 
\end{prop}
Then according to Theorem (\ref{main_theorem 1_ existance of canonical bases with same mutual position}) we state the
\begin{teor}\cite{Vac_isoclinic}  \label{Sp(n) orbit of a 4 dimensional isoclinic subpace}
The invariants $(\xi,\chi,\eta, \Delta)$ together with the angles $(\theta^I, \theta^J, \theta^K)$ determine the orbit of any $U \in \mathcal{IC}^4$. In particular if $U$ is orthogonal or 2-planes decomposable (in which case $(\Gamma,\Delta)=(1,0)$) the first set reduces to  the pair $(\xi, \chi)$.
\end{teor}

 A logical consequence of the above theorem is the following
 \begin{coro} \label{same invariants w.r.t. one admissible basis implies same invariants w.r.t. any admissible basis}
 If the  pair of subspaces $(U,W)$ share the same invariants  $(\xi,\chi,\eta,  \Gamma, \Delta)$  and the same angles of isoclinicity $(\theta^I, \theta^J, \theta^K)$ w.r.t. the admissible basis $(I,J,K)$ then they share the same invariants and angles w.r.t. any admissible basis.
 \end{coro}


\subsection{The 4-dimensional  complex subspace}

\begin{prop}\label{isoclinic complex 4-plane}
Let $(U,I)$ be a 4-dimensional complex  subspace. Then $U \in \mathcal{IC}^4$.  
Moreover for any pair $X,Y \in U$  belonging to orthogonal $I$-complex 2-planes (i.e. $L(X,IX) \perp L(Y,IY)$ ), the characteristic angle (i.e. the angle $\widehat{\mathcal{Q} X,\mathcal{Q} Y}$)
is an invariant of $U$ and equals the Euclidean angle of the pair $(U,KU)$.
\end{prop}

\begin{proof}
For any unitary pair $X,Y \in U$  with $Y \in L(X,IX)^\perp$, the set  $\{X,IX,Y,IY\}$ is an orthonormal basis of $U$. 
Let $(I,J,K)$ be an adapted basis.  Being $U=IU$, the pair $(U,IU)$ is isoclinic with cosine of the angle of isoclinicity equal to 1. Moreover,  from point (\ref{for any K in I perp  KU=JU is I-complex})of the Proposition (\ref{properties of some quternionic,complex and real subspaces}),  one has  $JU=KU$.
We compute the principal angles $\theta_i \in [0,\pi/2], \; i=1, \ldots,4$ between the $I$-complex 4-planes $U$ and $KU$ recalling that their square cosines are the eigenvalues of the symmetric matrix $G G^t$ where by $G$ we denote the Gram matrix  of $U \times KU$ (matrix of the orthogonal projector $Pr^U: KU \rightarrow U$). Such matrix, w.r.t. the orthonormal basis  $(X,IX,Y,IY)$ of $U$ and $(KX,JX,KY,JY)$ of $KU$, assumes the form
\[
G= \left(   \begin{array}{cccc}
KX & JX & KY & JY\\
\hline\\
\\
 0 & 0 &  <X,KY> &  <X,JY> \\
\\
 0 & 0 & <X,JY> & -<X,KY>\\
\\
-<X,KY> & -<X,JY> &  0 & 0 \\
\\
-<X,JY> & <X,KY> &0 & 0
\end{array}    \right)
\]
  Therefore, according to Proposition (\ref{general form for the matrix of projector omega^A of an isoclinic 4 dimensional subspace}), the pair  of 4-dimensional $I$-complex subspaces $(U,KU)$ is isoclinic and  according to the Proposition (\ref{isoclinicity w.r.t. one hypercomplex basis implies isoclinicity w.r.t. any compatible complex structure}) one has that $U \in \mathcal{IC}^4$.
Denoting by  $a=<X,KY>$ and $b=<X,JY>$, one has
\[
G^t G= \left(   \begin{array}{cccc}

a^2 + b^2 & 0 &  0 &  0 \\
\\
 0 & a^2 + b^2 & 0 & 0\\
\\
0 & 0 &  a^2 + b^2 & 0 \\
\\
0 & 0 & 0 & a^2 + b^2
\end{array}    \right).
\]
So, according to  the Definition (\ref{definition of isoclinicity}),   the angle of  isoclinicity $\theta^K$ of  the complex 4-planes $U$ and $KU$ or equivalently  one of the four identical
principal angles is given by
\[\cos \theta^K=\sqrt{ <X,KY>^2 + <X,JY>^2}.\] 

From Proposition (\ref{cosine of angle between subspaces is the product of the cosines of the principal angles}) one has
\begin{equation} \label{angle of isoclinicity}
\cos \theta_1 \cdot \cos \theta_2 \cdot  \cos \theta_3 \cdot \cos \theta_4= \cos \widehat{(U,KU)}=(<X,KY>^2 + <X,JY>^2)^2
\end{equation}
which clearly depends only on $U$ and $KU$. We deduce that
$(<X,JY>^2 +<X,KY>^2)$ is  an invariant of $U$ for any $I$-orthonormal pair $X,Y \in U$ (i.e. for any $Y \in L(X,IX)^\perp \cap U$).
This implies that, for any $I$-orthonormal pair $X,Y \in U$, 
the characteristic angle $\theta=\widehat{\mathcal{Q} X,\mathcal{Q} Y}$
is an invariant of $U$ as well. In fact, if  $X'=AX, \; Y'= AY$ 
where $A$ is an  orthogonal map commuting with $I$ in order to preserve the $I$-orthogonality between $X',Y'$,
\begin{equation} \label{characteristic angle}
\begin{array}{ll} \cos \widehat {\mathcal{Q} X,\mathcal{Q} Y}&=  (<X,Y>^2 + <X,IY>^2 +<X,JY>^2 +<X,KY>^2)^2=(<X,JY>^2 +<X,KY>^2)^2\\
&= \cos \widehat{(U,KU)}= (<X',JY'>^2 +<X',KY'>^2)^2=\cos \widehat {\mathcal{Q}X',\mathcal{Q}Y'}.
\end{array}
\end{equation}

\end{proof}

\begin{defi}
Let $(U,I)$ be a 4-dimensional $I$-complex subspace. We call 
\textbf{$I^\perp$-K\"{a}hler angle} of a 4-dimensional $I$-complex subspace $U$,  and denote it by $\theta^{I^\perp}(U)$, the angle of isoclinicity of the pair $(U,KU)$, for any $K \in I^\perp$.   
 One has
\[
\cos \theta^{I^\perp}(U)= \sqrt{ <X,KY>^2 + <X,JY>^2}= \cos \theta^J= \cos \theta^K
\]
 where $(I,J,K)$ is any adapted basis and $(X,Y)$ unitary with $Y \in L(X,IX)^\perp \cap U$. We denote  such subspace by the triple $(U,I,\theta^{I^\perp})$.
\end{defi}



Let consider the case that $(U,I)$ is not totally complex i.e. $\theta^K \neq \pi/2$.
it is easily seen that given the unitary vectors $X,Y$ of $U$ with $Y \in L(X,IX)^\perp \cap U$,  
and denoted by   $a= <X,KY>$ and $b= <X,JY>$, the unitary vector $Z_2 \in U$ such that
 $X$ and $KZ_2$ are associated left and right singular vectors of the pair $(U,KU)$ i.e.  $Pr^{KU} X= \cos \theta^K K Z_2$ and  $Pr^U (K Z_2)= \cos \theta^K X$ is given by
\begin{equation}\label{singular vector Z} Z_2= \frac{a}{\cos \theta^K}Y +  \frac{b}{\cos \theta^K}IY.
\end{equation}
 Clearly it  does not depend on the orthonormal pair $(Y,IY)$;  it belongs to the $\omega^I$-standard $I$-complex 2-plane $L(X,IX)^\perp=L(Y,IY)$.
We can then state the following
\begin{coro} \label{related singular vectors are I-orthogonal}
Consider $(U,I,\theta^{I^\perp}\neq \pi/2)$. 
For any complex structure $K \in  \mathcal{Q}$ anticommuting with $I$, the orthogonal projection $Pr^{KU} X \in KU$ of a unitary vector $X \in U$ is such that the  unitary vectors  $(X,Z_2=\frac{K^{-1}Pr^{KU}X}
{\cos \theta^K})$
belong to strictly orthogonal $I$-complex 2-planes.
\end{coro}

Clearly $(X,KZ_2)$ are  related principal vectors of the pair of  4-dimensional $I$-complex subspaces $(U,KU)$.  Considering a different complex structure $\bar K \in I^\perp$, the
   vector $Z_2= \frac{\bar K^{-1} Pr^{\bar K U}X}{\cos \theta^K}$ changes inside the $I$-complex 2-plane   $L(X,IX)^\perp \cap U$.
In particular one has that, for any adapted basis $(I,J,K)$, the  vector   $Y_2=\frac{J^{-1}Pr^{KU}(X)}{\cos \theta^J}=-IZ_2$  and $Z_2= \frac{K^{-1} Pr^{K U}X}{\cos \theta^K}$
form an orthonormal basis of $L(X,IX)^\perp$.



From the isoclinicity of $U$ and $KU$ any orthonormal basis in $U$ is a  basis of  singular vectors as well as any orthonormal basis in $KU$ is a basis of right singular vectors.
Clearly to any such basis in $U$ corresponds only one basis of right singular vector in $KU$ in order to have a pair of \textit{related bases} (\cite{Riz1}).

With respect to these  related bases, the Gram matrix $G(U \times KU)$
is diagonal with non negative diagonal entries equal to $\cos \theta^K=\cos \theta^J= \cos \theta^{I^\perp}$.  The existence of such diagonal form is stated  in \cite{Af} by   Theorem (\ref{Afriat Theorem})    and  follows from the theory of the singular values decomposition applied to the matrix $G$.

In case instead  $(U,I)$  is totally complex, for any $X \in U$, one has that $Pr^{KU} X= \bf{0}$.
Being the $\omega^I$ standard 2-plane with  leading vector $X_1$ given by $L(X_1,X_2=-IX_1)$, in this case we assume that $Z_2=X_2$ for any $K \in I^\perp$


Observe that 
 the cosine of the characteristic angle $\cos \widehat {\mathcal{Q}X,\mathcal{Q}Y}$
  equals the square cosine of the angle between the $I$-complex planes $\C X=L(X,IX)$ and $\C (KY)=L(KY,JY)$
  i.e. for any $X,Y \in U$ such that $\C X \perp \C Y$ it is
\[\cos \widehat {\mathcal{Q} X,\mathcal{Q} Y}=\cos \widehat{(U,KU)}=(<X,KY>^2 + <X,JY>^2)^2=\cos^2 \widehat{(\C X, \C (KY))}\]
for any adapted basis $(I,J,K)$. 
In particular if $(U,I)$ is totally complex one has that  $\cos \widehat {\mathcal{Q} X,\mathcal{Q} Y}=\cos \widehat{(U,KU)}=0$.


 Recalling the definition of $S$ introduced in Proposition (\ref{$S$ sum of square cosined of the angles of isoclinicity (U,IU),(U,JU),(U,KU) of isoclinic subspaces}), one has the following
\begin{prop}
 There exists a 1:1 correspondence between the  characteristic deviation $\Delta(U)$  of a 4-dimensional $I$-complex subspace $(U,I, \theta^{I^\perp})$
 and the $I^{\perp}$-K\"{a}hler angle $\theta^{I^\perp} \in [0,\pi /2]$. It is given by
\begin{equation} \label{characteristic deviation of a 4-dimensional complex subspace}
\Delta(U,I,\theta^{I^\perp})=\frac{2 \cos^2 \theta +1}{3}=\frac{\cos^2(U,IU) + \cos^2(U,JU) + \cos^2(U,KU)}{3}= \frac{S}{3}
\end{equation}
\end{prop}
\begin{proof}
If follows from (\ref{characteristic deviation of a subspace}). Consider the orthonormal  basis $(X,Z_2,IX,IZ_2)$ where the pair $(X,KZ_2)$ are related left and right singular vectors of $(U,KU)$. One has
\[
\begin{array}{lll}
\Delta(U,I,\theta^{I^\perp})&= & \frac{1}{6}( \Delta(X,Z_2)_\R +\Delta(X,IX)_\R +\Delta(X,IZ_2)_\R +\Delta(Z_2,IX)_\R +\Delta(Z_2,IZ_2)_\R +\Delta(IX,IZ_2)_\R)=\\
& = & \frac{1}{6}(<X,KZ_2>^2 + 1 +  <X,KZ_2>^2 + <X,KZ_2>^2 + 1  + <X,KZ_2>^2)=\frac{2 \cos^2 \theta^{I^\perp} +1}{3}= \frac{S}{3} 
\end{array}
\]
\end{proof}

\newpage
\subsection{The associated plane of a  4-dimensional complex subspace}

The standard form of the skew-symmetric form $\omega^K: (X \times Y) \rightarrow <X,KY>$ restricted to the 4-dimensional pure subspace $(U,I, \theta^{I^\perp})$, for any $K \in I^\perp$, determines a decomposition $U=U_1 \stackrel{\perp} \oplus  U_2$ into an orthogonal  sum of a pair of $K$-orthogonal $\omega^K$-standard 2-planes (i.e. $U_1 \perp KU_2$). If $U_1=L(X_1,X_2)$ and $U_2=L(X_3,X_4)$ where the bases are orthonormal, clearly  $(X_1,X_2,X_3,X_4)$ and $(KX_2,-KX_1,KX_4,-KX_3)$  are  bases of related left and right singular vectors of the projector $Pr^{KU}:U \rightarrow KU$
or equivalently of related principal vectors of the pair of subspaces $(U,KU)$.

Let now consider the case that  the 4-dimensional pure subspace $(U,I)$ is not totally complex. In this case,
from the Corollary (\ref{related singular vectors are I-orthogonal}) it is clearly $U_2=IU_1$ and consequently $U_1 \perp IU_1$. Then, if $J$ completes the adapted basis $(I,J,K)$,  it is also  $U_1 \perp JU_1$.
Therefore, for any $K' \in I^\perp$, given an $\omega^{K'}$-standard 2-plane $U_1 \subset U$ one has that $U=U_1 \stackrel{\perp} \oplus  IU_12$   is an $\omega^{K'}$-standard decomposition of $U$ into orthogonal standard 2-planes.

\begin{defi}
 We call \textbf{associated plane} to the  $I$-complex 4-dimensional pure subspace $(U,I,\theta^{I^\perp}\neq \pi/2)$ 
any 2-dimensional plane $U' \subset U$ characterized by the existence of an orthonormal basis,  say   $(X,Z)$,  such that the  vectors  $(X,KZ)$ are a pair of related principal vectors of the  pair $(U,KU)$  for some  $ K \in I^\perp$.
\end{defi}

If any such basis exists, all  consistently oriented orthonormal bases have the same property. Clearly $<X,KZ>= \cos \theta^{I^\perp}$. Being $U$ pure, one has that   $U'$ is  necessarily totally real. 
From the Corollary (\ref{related singular vectors are I-orthogonal}), $Z(K)= \frac{K^{-1}Pr^{KU}X}{\cos \theta^K} \in    L(X,IX)^\perp$.\\
For any $X \in U$ and $K \in I^\perp$, all subspaces $U'(K)= L(X,Z(K))$  where  $Z(K)= \frac{K^{-1}Pr^{KU}X}{\cos \theta^K}$ are associated plane of $(U,I,\theta^{I^\perp}\neq \pi/2)$.
When we need to specify the structure $K \in S(\mathcal{Q})$, we denote the associated plane $U'$ as  $U'(K)$.


\begin{prop} \label{equivalent definitions of associated planes}
Let  $(U,I,\theta^{I^\perp})$ be a 4-dimensional 
$I$-complex subspace with $I^\perp$-K\"{a}hler angle $\theta^{I^\perp} \neq \pi/2$. The  2-plane  $U'=L(X,Z) \subset U$
is an associated plane iff any of the following equivalent conditions are satisfied.


\begin{enumerate}
\item  \label{la proiezione di un piano associato coincide con la sua $K$-immegine} There exists $K \in I^\perp$ such that $Pr_U^{KU}U'=KU'$.
\begin{proof}
If $U'(K)$ is an associated plane and $(X,Z)$ an orthonormal basis such that the pair $(X,KZ)$ are related principal vectors of the pair $(U,KU)$, then
$Pr^{KU} X=  \cos \theta^{I^\perp} KZ$  and by the skew-symmetry of $\omega^K$ one has $Pr^{KU} Z=  \cos \theta^{I^\perp} (-KX)$.  Then  $Pr^{KU}U'=KU'$.
 Viceversa, if there exists some $K \in I^\perp$ such that $Pr^{KU}U'=KU'$, let $(X,Y)$ be an orthonormal basis of $U'$. One has that $Pr^{KU} X=<X,KY> KY$ with clearly $<X,KY>=\pm \cos \theta^{I^\perp}$. According to the sign, either $(X,KY)$ or $(Y,KX)$ are   a pair of related principal vectors of the pair $(U,KU)$.
\end{proof}

\item $(U,I,\theta^{I^\perp})= U'   \stackrel{\perp} \oplus IU'$. 
\begin{proof}
Let $U'=U'(K)=L(X,Z)$ be an associated plane of  $U$. Clearly also $IU'=IU'(K)$ is an associated plane. Furthermore, by  the Corollary (\ref{related singular vectors are I-orthogonal}), the vectors $(X,Z,IX,IZ)$ form an orthonormal basis which implies that the direct sum $U=U' \oplus IU'$ is orthogonal. 
Viceversa, let $(U,I,\theta^{I^\perp})= U'   \stackrel{\perp} \oplus IU'$. By the Proposition (\ref{properties of some quternionic,complex and real subspaces})  point (\ref{A complex subspace does not contain a complex subspace by a different complex structure}), the 2-plane $U'=L(X,Z)$ is totally real.
  W.r.t. the orthonormal basis $(X,Z,IX,IZ)$ one has  $Pr^{KU}X= <X,JZ>Z + <X,KZ>IZ$ then $<X,K'Z>= \cos \theta^{I^\perp}$ for $K'=\frac{1}{\cos \theta^{I^\perp}}<X,JZ> J +<X,KZ> K$.
 \end{proof}

\item   There exists some $K \in I^\perp$ w.r.t. which $U'$ is a standard 2-plane of $\omega^K|_U$.
\begin{proof}
  If $U'(K)$ is an associated plane then, w.r.t. to the decomposition $U=U' \stackrel{\perp }\oplus IU'$,  $\omega^K|_U$ assumes standard form.
Viceversa in the decomposition $U=U_1 \oplus IU_1$ associated to the form $\omega^K|_U$, the standard 2-plane  $U_1$ is  an associated plane of $(U,I,\theta^{I^\perp})$ since $Pr^{KU}U_1=KU_1$ and the conclusion follows from point (\ref{la proiezione di un piano associato coincide con la sua $K$-immegine}).
\end{proof}

\item There exists $K \in I^\perp$ such that the  $K$-K\"{a}hler angle of $U'$ equals the $I^\perp$-K\"{a}hler angle $\theta^{I^\perp}$.
\begin{proof}
By the isoclinicity of the pair $(U,KU)$, the angle between any vector $X \in U$ and the subspace $KU$ equals the $I^\perp$-K\"{a}hler angle $\theta^{I^\perp}$.
if $U'(K)$ is an associated plane, from 1)  $Pr^{KU}U'=KU'$, then both singular values of the projector $Pr^{KU}$ restricted to $U'$ equal $\cos \theta^{I^\perp}$. Viceversa if there exists $K \in I^\perp$ such that the  $K$-K\"{a}hler angle of $U'$ equals  $\theta^{I^\perp}$  clearly  $Pr^{KU}U'=KU'$ i.e.,  from previous point (\ref{la proiezione di un piano associato coincide con la sua $K$-immegine}), $U'$  is an associated plane.
\end{proof}

\item  \label{Immaginary measure of an associated plane} $\mathcal{IM}(U')= aj+ bk$ where $j,k \in i^\perp$ and orthonormal  with $a^2+ b^2= \cos^2 \theta^{I^\perp}$  i.e. iff $\mathcal{IM}(U') \in i^\perp, \; \Delta(U')= a^2+ b^2= \cos^2 \theta^{I^\perp}$.
\begin{proof}
If $U'(K)=L(X,Z)$ is an associated plane  of  $(U,I,\theta^{I^\perp})$  with $<X,KZ>= \cos \theta^{I^\perp}$, then  $U' \perp IU',\; U' \perp JU'$  where $J$ completes the adapted basis $(I,J,K)$.
 After identifying the adapted  hypercomplex  structures  $(I,J,K)$ of $V$  with $(R_{-i},R_{-j},R_{-k})$ of $\mathbb{H}$, one has   $\mathcal{IM}(U')= <X,IZ>i + <X,JZ>j + <X,KZ>k  = \cos \theta k$.  Using a different adapted basis the results follows.

Viceversa if $U' \subset (U,I, \theta^{I^\perp})$, with  $\mathcal{IM}(U')= aj+ bk$ w.r.t. the adapted basis $(I,J,K)$, one has that $U'$ is an associated plane w.r.t. $K'= aJ + bK, \; a^2 + b^2=1$. In fact  $<X,K'Z>= \cos \theta^{I^\perp}$. In case $<X,KZ > <0$ we consider the orthonormal basis $(X,-Z)$ or any other with the same orientation.
    From 3),  $U'(K')$ is an associated plane.
\end{proof}

\item One and hence any orthonormal basis of $U'$, say $(X,Z)$ is $I$-orthogonal i.e. $Z \in L(X,IX)^\perp$. 
\begin{proof}
If $U'(K)=L(X,Z)$ is an associated plane then for any $X \in U'$ from the Corollary (\ref{related singular vectors are I-orthogonal}), one has $Z=\frac{K^{-1}Pr^{KU}X}{\cos \theta^{K}} \in L(X,-IX)^\perp$. Then $Pr^{KU}(U')= KU'$ and the conclusion follows from point (\ref{la proiezione di un piano associato coincide con la sua $K$-immegine}).
Viceversa if the unitary basis $(X,Z)$  is $I$-orthogonal, then $(X,Z,IX,IZ)$ is an orthonormal basis of $U$  and $(KX,KZ,JX,JZ)$ of $KU$ w.r.t. which $Pr^{KU}X= <X,KZ>KZ+ <X,JZ>JZ$ with $<X,KZ>^2+ <X,JZ>^2= \cos ^2 \theta^{I^\perp}$. Then $U'=U'(K')$ with $K'= \frac{1}{\cos \theta^{I^\perp}} <X,KZ>K + <X,JZ>J$ is an associated plane  since $<X,K'Z>=\cos \theta^{I^\perp}$.
\end{proof}

\end{enumerate}
\end{prop}

We summarize the above characterizations of the associated subspace in the following

\begin{prop} \label{angle of isoclinicity of an associated plane}
Let $(U,I,\theta^{I^\perp} \neq \pi/2)$ be a 4-dimensional pure $I$-complex subspace. Then, for any $X \in U$,  and $K \in I^\perp$ one has  $U=U_1(K) \oplus IU_1(K)$ is direct orthogonal sum of the uniquely defined  $K$-orthogonal  associated planes $U_1=L(X,Z)$ and $U_2=IU_1=L(IZ,IX)$ where $Z= \frac{K^{-1} Pr_U^{KU} (X)}{\cos \theta^{I^\perp}}$. The $I^\perp$-K\"{a}hler angle $\theta^{I^\perp}$  is the same as the $K$-K\"{a}hler angle  $\Theta^{K}(U_1)=\Theta^{K}(U_2)$ of the associated planes $U_1$ and  $U_2=IU_1$ i.e.
\[\cos \theta^{I^\perp}= \cos \Theta^{K}(U_1)=  <X,KZ>=<IZ,KIX>=\cos \Theta^{K}(U_2).\]
Moreover  
\[\mathcal{IM}(U_1(K))=\mathcal{IM}(U_2(K))= \cos \theta k=\cos \theta^{I^\perp}.\]
\end{prop}


Observe that  the strictly orthogonal associated planes  $U=U_1(K)$ and $IU_1= U_2(K)$ 
are  only $K$-orthogonal but clearly never $I$ orthogonal (furthermore they are not  $J$ orthogonal unless $U$ is totally complex i.e. $\cos \theta^K=0$).  In particular they are never orthogonal in Hermitian sense. 
Namely, since any  totally real 2-plane   never belong to a quaternionic line but to a quaternionic 2-plane, we have that  the pair of quaternionic planes containing $U_1$ and $U_2$ (eventually coinciding) are never orthogonal.  We conclude this section with the

\begin{prop} \label{same quaternionic angle of 4-dimensional complex subspaces}
  Given $(U,I,\theta^{I^\perp})$ and the associated plane $U_1(K'),\; K' \in I^\perp$ of $(U,I,\theta^{I^\perp})$, the $A^\perp$-K\"{a}hler angle $\theta^{A^\perp}$ of the $A$-complex subspace $\bar U= U_1(K') \oplus AU_1(K')$ with $A= aI+  b(\alpha J + \beta K), \; a^2 + b^2= \alpha^2 + \beta^2=1$  equals $\theta^{I^\perp}$.
\end{prop}
\begin{proof}
It follows from the fact that for $A \in {K'}^\perp$, $U_1(K')$ is an associated plane  of $\bar U$ as well. In fact,  let  $(I,J,K)$ be an  adapted basis and consider for instance $U_1(K)$. Being $U_1(K) \perp IU_1(K)$ and $U_1(K) \perp JU_1(K)$ one has that  $U_1(K) \perp J'U_1(K), \; \forall J' \in L(I,J)$. Consequently $Pr^{KU} U_1(K)= KU_1(K)$. Extending such results to all $K' \in L(J,K)$ the conclusion follows.
\end{proof}

Then, given  $(U,I,\theta^{I^\perp})$, for any $A \in S(\mathcal{Q})$ we can build an $A$-complex 4-dimensional subspace with  $\theta^{A^\perp}= \theta^{I^\perp}$.

As an example, let consider $I'= \frac{1}{3} I + \frac{2}{3}J + \frac{2}{3}K$. Then $I' \in L(I,J')$ with $J'= \frac{1}{\frac{2 \sqrt{2}}{3}} (\frac{2}{3}  J + \frac{2}{3} K)$, In this case,  $a=\frac{1}{3}$, $b=\frac{2 \sqrt{2}}{3}$, $\alpha=\beta= \frac{1}{\sqrt{2}}$. The complex structure in $L(J,K)$ orthogonal to $J'$ is $K'= \frac{1}{\frac{2 \sqrt{2}}{3}}(-\frac{2}{3} J + \frac{2}{3} K)$ and consequently $\bar Z_2=\frac{{K'}^{-1}Pr^{KU}X_1}{\cos \theta^{I^\perp}}=  \frac{1}{\frac{2 \sqrt{2}}{3}}(\frac{2}{3} IZ + \frac{2}{3} Z)=\frac{1}{\sqrt{2}}(IZ + Z)$ and the associated plane is $U(K')=L(X_1,\bar Z_2)$.
It is straightforward to  verify that $\bar U= U_1(K') \oplus I'U_1(K')$ is $I'$-complex with ${I'}^\perp=I^\perp$ i.e. $\bar U=(\bar U,I',\theta^{I^\perp})$. In fact,
considering the adapted basis $(I',\bar J, \bar K)$ with $\bar J= \frac{2}{3} I + \frac{1}{3}J  -\frac{2}{3}K$ and  $\bar K= -\frac{2}{3} I + \frac{2}{3}J  -\frac{1}{3}K$ one has
\[
\begin{array}{l}
<X, \bar K \bar Z_2>= <X,(-\frac{2}{3} I + \frac{2}{3}J  -\frac{1}{3}K)(\frac{1}{\sqrt{2}}(IZ + Z))> \frac {1}{\cos \theta^{I^\perp}}= -\frac{1}{\sqrt{2}}<X_1,KZ>=-\frac{1}{\sqrt{2}}\cos \theta^{I^\perp}\\
<X,\bar J \bar Z_2>= <X,(\frac{2}{3} I + \frac{1}{3}J + -\frac{2}{3}K)(\frac{1}{\sqrt{2}}(IZ + Z))> \frac {1}{\cos \theta^{I^\perp}}= -\frac{1}{\sqrt{2}}<X_1,KZ>=-\frac{1}{\sqrt{2}}\cos \theta^{I^\perp}\\
\end{array}
\]
which implies that $\cos \theta^{I'^\perp}= \sqrt{<X,\bar J \bar Z_2>^2 + <X,\bar K \bar Z_2>^2}=\cos \theta^{I^\perp}$.

What stated in Proposition (\ref{same quaternionic angle of 4-dimensional complex subspaces}) will be relevant when studying the $Sp(n) \cdot Sp(1)$-orbits in the real Grassmannian. In an article that we will publish soon we will show that the $I^\perp$-K\"{a}hler angle $\theta^{I^\perp}$  of  an $I$-complex  4-dimensional subspace of a quaternionic Hermitian vector space constitutes the full system of invariant for its $Sp(n) \cdot Sp(1)$-orbit. Then, from  Proposition (\ref{same quaternionic angle of 4-dimensional complex subspaces}), we have that,
given a  2-plane $U'$  with  $\mathcal{IM}(U')= \cos \theta k$, which as stated in point (\ref{Immaginary measure of an associated plane}) of the Proposition (\ref{equivalent definitions of associated planes}) is an associated plane of $U=U' \oplus IU'$,
all subspaces $U_A=U' \oplus AU'$ with  $A= aI+  b(\alpha J + \beta K), \; a^2 + b^2= \alpha^2 + \beta^2=1$  are in the same $Sp(n) \cdot Sp(1)$-orbit.

\subsection{Canonical bases and canonical matrices of a 4-dimensional complex subspace}

Let $(U,I,\theta^{I^\perp})$ be a 4-dimensional $I$-complex subspace and $(I,J,K)$ an adapted basis. Recall that $JU=KU$.
Using the same notations that appear in \cite{Vac_isoclinic},  for any  unitary $X_1 \in U$, we denote by $X_2= I^{-1}Pr^{IU} X$, by $Y_2= \frac{J^{-1}Pr^{JU} X}{\cos \theta^{I^\perp}}$ and by  $Z_2= \frac{K^{-1}Pr^{KU} X} {\cos \theta^{I^\perp}}$.
Clearly $X_2=-IX_1$. The pair $(X_1,-IX_1)$  is an  $\omega^I$-standard  basis   of    
 $U_1=L(X,-IX)$. 
 Clearly  $Pr^{IU}U_1= IU_1=U_1$. 
Furthermore we denote by $\xi=<X_2,Y_2>, \quad \chi=<X_2,Z_2>, \quad \eta=<Y_2,Z_2>$. According to  the Proposition (\ref{invariance of <X_2,Y_2>}) such triple is an invariant of $U$.

\begin{prop} \label{the vectors $X_2,Y_2,Z2$ for an orthonormal nasis of thge orthogonal complement of $X_1$}
Let $(U,I,\theta^{I^\perp})$ be a 4-dimensional $I$-complex subspace not totally complex and $(I,J,K)$ an adapted basis. Choose $X_1 \in U$ unitary and let $U_1=L(X_1,-IX_1)$ be the $\omega^I$-standard plane. The vectors
$ Y_2= \frac{J^{-1}Pr^{JU} X}{\cos \theta^{I^\perp}} \in U_1^\perp, \quad    Z_2= \frac{K^{-1}Pr^{KU} X} {\cos \theta^{I^\perp}}  \in U_1^\perp$. Moreover one has $\xi=\chi=\eta=0$ and consequently
\[(X_1, X_2= I^{-1}Pr^{IU} X=-IX_1, \quad  Y_2= \frac{J^{-1}Pr^{JU} X}{\cos \theta^{I^\perp}}, \quad    Z_2= \frac{K^{-1}Pr^{KU} X} {\cos \theta^{I^\perp}})\]
 is an orthonormal basis of $U$ and, upon reordering,  
  they form a triple of 
 standard bases $\{X_i \}, \{Y_i \},\{Z_i \}$ centered on $X_1$ of respectively $\omega^I|_U,\omega^J|_U,\omega^K|_U$. Furthermore such triple are exactly the $\{X_i\}, \{Y_i\}, \{Z_i\}$ chains  centered on $X_1$ defined in (\ref{associated chains}).
W.r.t. such chains the  orthogonal matrices  $C_{IJ}={<X_i,Y_j>}, \; C_{IK}={<X_i,Z_j>}$  are given in (\ref{canonical matrices of a 4 dimensional complex subspace}) and, as stated in the Proposition (\ref{invariance of $C_{IJ}$ and $C_{IK}$ w.r.t. any leading vector}), are invariants of $(U,I,\theta^{I^\perp})$.
 Furthermore they do not depend on the adapted basis. 
\end{prop}

\begin{proof}
From Corollary (\ref{related singular vectors are I-orthogonal}) one has that the  unitary vectors $Y_2= \frac{J^{-1}Pr^{JU} X}{\cos \theta^J}  \in U_1^\perp$ and $Z_2= \frac{K^{-1}Pr^{KU} X}{\cos \theta^K}  \in U_1^\perp$ which implies that $\xi=<X_2,Y_2>=0$ as well as $\chi=<X_2,Z_2>=0$.  Moreover $Y_2= \frac{J^{-1}Pr^{JU}X_1}{\cos \theta^J}= \frac{J^{-1}KX_1 \cos \theta^K}{\cos \theta^J}=-IZ_2$ which implies that   $\eta= <Y_2,Z_2>=0$.
To obtain the chains $\{X_i \},\{Y_i \},\{\tilde X_i,\{Z_i \}\}$ it is straightforward to verify that
\[X_4=Y_2, \; Y_4= -X_2, \; \tilde X_4= Z_2,  \; Z_4= -X_2.\]
Furthermore \[\tilde X_3= -I^{-1} Pr^{IU}\tilde X_4=I\tilde X_4= IZ_2.\]
\[Y_2=\frac{J^{-1}Pr^{JU}X_1}{\cos \theta}=\frac{J^{-1}Pr^{KU}X_1}{\cos \theta}=J^{-1}KZ_2= -IZ_2.\] Then
\[ \Delta= <X_4,\tilde X_3>= -<IZ_2,IZ_2>=-1.\]

The  chains $\{X_i \},\{Y_i \},\{\tilde X_i \},\{Z_i \}$ of an $I$-complex subspace with leading vector $X_1 \in U$ w.r.t. the adapted basis  $(I,J,K)$ are:
\begin{equation} \label{the chains of a 4dimensionalcomplex subspace}
\begin{array}{lllllll}
\{X_i \} &=&\{X_1,X_2,X_3,X_4 \}&=& \{X_1,-IX_1,Z_2,-IZ_2 \}&=&\{X_1,X_2,Z_2,Y_2\} \\
\{Y_i \}&=&\{X_1,Y_2,X_3,Y_4 \}&=&\{X_1,-IZ_2,Z_2,IX_1 \}&=&\{X_1,Y_2,Z_2,-X_2\}\\
\{\tilde X_i \}&=&\{X_1,X_2,\tilde X_3,\tilde X_4 \}&=&\{X_1,-IX_1,IZ_2,Z_2 \}&=&\{X_1,X_2,-Y_2,Z_2 \}\\
\{Z_i \}&=&\{X_1,Z_2,X_3,Z_4 \}&=&\{X_1,Z_2,IZ_2,IX_1\}&=&\{X_1,Z_2,-Y_2,-X_2\}
\end{array}
\end{equation}

Therefore the set $(\xi,\chi,\eta,\Gamma,\Delta)=(0,0,0,0,-1)$ is an  invariant (resp. an intrinsic property) of an $I$-complex subspace (resp. quaternionic subspace).

In  particular for a quaternionic subspace it is $Y_2= -JX_1$ and $Z_2=-KX_1$ then
\[
\begin{array}{lll}
\{X_i \} &=& \{X_1,-IX_1,-KX_1,-JX_1 \}\\
\{Y_i \}&=&\{X_1,-JX_1,-KX_1,IX_1 \}\\
\{\tilde X_i \}&=&\{X_1,-IX_1,JX_1,-KX_1 \} \\
\{Z_i \}&=&\{X_1,-KX_1,JX_1,IX_1 \}
\end{array}
\]

W.r.t. the canonical bases $\{X_i \},\{Y_i \},\{Z_i \}$  the canonical matrices
(\ref{canonical matrices $C_{IJ}$ $C_{IK}$ and of 4-dimensional isoclinic subspace w.r.t. the associated chains})
 of an $I$-complex 4-dimensional subspace not totally complex are
 \begin{equation} \label{canonical matrices of a 4 dimensional complex subspace}
C_{IJ}=C'_{IK}
=\left(
\begin{array} {cccc}
1 & 0 & 0 & 0\\
0 & 0 & 0 & -1 \\
0 & 0 & 1 & 0 \\
0 & 1 & 0 & 0
\end{array}
\right),
\qquad \qquad
  C_{IK}=\left(
\begin{array} {cccc}
1 & 0 & 0 & 0\\
0 & 0 & 0 & -1\\
0 & 1   & 0 & 0 \\
0 & 0 & -1 & 0
\end{array}
\right).
\end{equation}
\end{proof}

In case  the 4-dimensional $I$-complex subspace $(U,I)$ is totally complex i.e. $U=(U,I,\pi/2)$ we are in a case of double orthogonality. In this case   we can always assume $X_2=Y_2=Z_2$ (see \cite{Vac_isoclinic}) and consequently $\xi=\chi=\eta=1$. From the Definition (\ref{The chains if all the 3 invariants is one or minus one}) one has $\{X_i \}=\{Y_i \}=\{\tilde X_i \}=\{Z_i \}$ and consequently $C_{IJ}=C_{IK}=Id$ as stated in the Corollary (\ref{some of the pair $(U,IU), (U,JU),(U,KU)$ is strictly orthogonal}). In this case clearly $U$ is a 2-planes decomposable subspace.

\vskip .3cm
Let denote by $Gr_{(I,\theta^{I^\perp})}^\R(4,4n)$ the subset  of 4-dimensional $I$-complex subspaces of $V$   with $I^\perp$-K\"{a}hler angle $\theta^{I^\perp}$ of the Grassmannian $G^\R(4,4n)$.

\begin{teor} \label{orbit of a complex 4-plane under Sp(n)}
Let $(U,I,\theta^{I^\perp}) \in Gr_{(I,\theta^{I^\perp})}^\R(4,4n)$.   The pair $(I,\theta^{I^\perp})$ composed by the complex structure $I \in S(\mathcal{Q})$  and the $I^\perp$-K\"{a}hler angle $\theta^{I^\perp}$ determines completely the $Sp(n)$-orbit of  $U$  in the Grassmannian $Gr^\R(4,4n)$ i.e.  the group $Sp(n)$ acts transitively on $Gr_{(I,\theta^{I^\perp})}^\R(4,4n)$.
In particular then all totally complex subspaces  form one $Sp(n)$-orbit in $Gr^\R(4,4n)$.
\end{teor}

\begin{proof}
The proof follows from  the Theorem (\ref{Sp(n) orbit of a 4 dimensional isoclinic subpace}) and the Proposition (\ref{canonical matrices of any  subspace of ${IC}^(4)$}). In fact in case $(U,I,\theta^{I^\perp} \neq \pi/2)$  one bas $(\xi,\chi,\eta, \Delta)=(0,0,0,-1)$ and the angles of isoclinicity  are $(0,{I^\perp},{I^\perp})$.

On the other hand, all totally 4 dimensional $I$-complex subspace are characterized by  $(\xi,\chi,\eta, \Delta)=(1,1,1,0)$ and $(\theta^I,\theta^J,\theta^K)=(0,\pi/2,\pi/2)$.
\end{proof}

We terminate the analysis of the 4-dimensional complex subspaces with the
\begin{prop} \label{xi=chi=eta=0 of a complex 4-plane  only w.r.t. an adapted basis}
Let $U$ be a  4-dimensional complex subspace  not totally complex. Then  $\xi=\chi=\eta=0$  only w.r.t. an adapted basis.
\end{prop}

\begin{proof}
 Let consider an $I$-complex subspace  $(U,I,\theta)$ and  a hypercomplex basis $(I',J',K')$ with $I= \alpha_1 I' + \beta_1 J' + \gamma_1 K'$. Then
\[I'= \alpha_1 I + \ldots, \qquad J'= \beta_1 I + \ldots, \qquad K'=\gamma_1 I + \ldots\]
 and square cosine of the  angles of isoclinicity  $\cos^2 \theta^{I'} , \cos^2 \theta^{J'} , \cos^2 \theta^{K'}$ between the pairs $(U,I'U),(U,J'U),(U,K'U)$.
 From  Proposition (\ref{isoclinicity of a 4-dimensional I-complex subspace U and I'U}) we have
\[
\begin{array} {llll}
\cos \widehat{(U,I'U)} & = \cos^2 \theta^{I'} &=& \alpha_1^2  \sin^2 \theta + \cos^2 \theta \\
\cos \widehat{(U,J'U)} & = \cos^2 \theta^{J'}&=& \beta_1^2  \sin^2 \theta + \cos^2 \theta \\
\cos \widehat{(U,K'U)} & = \cos^2 \theta^{K'}&=& \gamma_1^2  \sin^2 \theta + \cos^2 \theta.
\end{array}
 \]
We can verify that   $S=\cos^2 \theta^{I'} + \cos^2 \theta^{J'} + \cos^2 \theta^{K'}= 1 + 2 \cos^2 \theta =3 \Delta(U)$.

From (\ref{equivalent form for the angle of isoclinicity of a 4 dimensional subspace isoclinic w.r.t. an hypercomplex basis}),
one has
\[
\begin{array}{lll}
\cos \theta^{I}=1 &=& \alpha_1^2(\alpha_1^2  \sin^2 \theta + \cos^2 \theta)     + \beta_1^2(\beta_1^2  \sin^2 \theta + \cos^2 \theta)     + \gamma_1^2(\gamma_1^2  \sin^2 \theta + \cos^2 \theta) +\\
& & 2  <X_2,Y_2> \alpha_1 \beta_1 \cos \theta^{I'} \cos \theta^{J'} + 2 <X_2,Z_2> \alpha_1 \gamma_1 \cos \theta^{I'} \cos \theta^{K'}+ 2 <Y_2,Z_2> \beta_1 \gamma_1 \cos \theta^{J'} \cos \theta^{K'}
\end{array}
\]
Being
\[
(\alpha_1^4 + \beta_1^4 + \gamma_1^4) \sin^2 \theta + (\alpha_1^2 + \beta_1^2 + \gamma_1^2) \cos^2 \theta= 1-2 \sin^2 \theta(\alpha_1^2 \beta_1^2 + \alpha_1^2 \gamma_1^2 + \beta_1^2 \gamma_1^2)
\]
we get
\[
\begin{array}{l}
2 \alpha_1 \beta_1 ( \alpha_1 \beta_1 \sin^2 \theta  -  <X_2,Y_2>  \cos \theta^{I'} \cos \theta^{J'}) +\\
+ 2 \alpha_1 \gamma_1 ( \alpha_1 \gamma_1 \sin^2 \theta  -  <X_2,Z_2>  \cos \theta^{I'} \cos \theta^{K'}) + \\
+ 2 \beta_1 \gamma_1 ( \beta_1 \gamma_1 \sin^2 \theta  -  <Y_2,Z_2>  \cos \theta^{J'} \cos \theta^{K'})=0
\end{array}
\]
Then  $\xi=\chi=\eta=0$ if $ \sin^2 \theta(\alpha_1^2 \beta_1^2 + \alpha_1^2 \gamma_1^2 + \beta_1^2 \gamma_1^2 )=0$ that is either if
$\cos \theta=1$ which  implies  $U$ quaternionic or if
  two among $(\alpha_1,\beta_1,\gamma_1)$ are zero i.e. if  $I'=\pm I$ or $J'=\pm I$ or $K'=\pm I$.
\end{proof}

Applying the expression (\ref{equivalent form for the angle of isoclinicity of a 4 dimensional subspace isoclinic w.r.t. an hypercomplex basis}) for the determination of the angle of isoclinic of the pair $(U,AU)$ for $A \in S(\mathcal{Q})$, we conclude this section with  the following

\begin{coro} \label{isoclinicity of a 4-dimensional I-complex subspace U and I'U}
Given a 4-dimensional $I$-complex subspace $(U,I,\theta^{I^\perp})$ and the compatible complex structure
$A= \alpha I + \beta J + \gamma K$,  the cosine of the angle of isoclinicity $\theta^{A}$ 
between the pair of subspaces $U$ and $AU$ is equal to
\begin{equation} \label{angle of isoclinicity of a 4-dimensional complex subspace}
\cos \theta^{A}=\sqrt{\alpha^2 + (1-\alpha^2)\cos^2 \theta^{I^\perp}}.
\end{equation}
Then  $U$  is never orthogonal unless it is totally-complex.
\end{coro}


\subsection{Decomposition of a $2m$-dimensional pure complex subspace}
Given a $2m$-dimensional complex  subspace $(U,I)\subset V$, let consider  the skew-symmetric form $\omega^K: (X,Y) \rightarrow <X,KY>$ for $K \in I\perp \cap S(\mathcal{Q})$   and denote by $\omega^K|_U$ its restriction to  $U$.
\begin{prop} \label{bases of invariat 2-planes of omega are I-ortogonal}
 Let $(U,I)$ be a $2m$-dimensional $I$-complex subspace and consider the principal angles of the pair $(U,KU)$.
 A principal angles $\theta \neq \pi/2$   has multiplicity $4k$; if instead $\theta = \pi/2$ its multiplicity is $2k$.
 \end{prop}

\begin{proof}
Let $(U,I)$ be a $2m$-dimensional  pure $I$-complex subspace. Consider first the case that  $U$ is not totally complex and let $U_i$ be a standard 2-plane of $\omega^K|_U$ with the cosine of the angle of isoclinicity  $\cos \theta^K(U_i)$ of the pair $(U_i,KU_i) \neq 0$. Observe that assuming on $U_i$ the orientation induced by a standard basis one has $\cos \theta^K(U_i)= \cos \Theta^K(U_i)$.

Let $(X,Z)$ be an $\omega^K$-standard basis of $U_i$ and consider
 the $I$-complexification $\tilde U_i=U_i \oplus IU_i \subset U$. 
From the Proposition (\ref{equivalent definitions of associated planes}) the  generators $(X,Z,IX,IZ)$ are orthonormal which implies that the above direct sum is orthogonal and that
 $\{Z_i\}=(X,Z,IZ,IX)$ is the $\omega^K$-chain of $\tilde U_i$  centered on $X$. The conclusion follows observing that for the angles of isoclinicity one has  $\theta^K(U_i)=<X,KZ>=<IZ,KIX>=\theta^K(IU_i)$.

In case a principal angle  $\theta= \pi/2$, let $\bar U$ the associated $\omega^K$-standard subspace. From point (\ref{3}) of the Claim (\ref{properties of some quternionic,complex and real subspaces}), it is clearly totally $I$-complex. Then $\bar U$ is 2-planes decomposable and its dimension is necessarily $2k$.
\end{proof}






\begin{teor} \label{canonical decomposition of a complex subspace}
Any  pure $I$-complex subspace  $(U^{2m},I)$ admits a  decomposition into an \underline{Hermitian orthogonal} sum of 4-dimensional  pure  $I$-complex subspaces plus, in case the dimension is not multiple of 4, an Hermitian orthogonal  (totally) $I$-complex 2-plane.
\end{teor}
\begin{proof}
From Proposition (\ref{bases of invariat 2-planes of omega are I-ortogonal}) one has that each principal angles $\theta \neq \pi/2$ between the pair $(U,KU)$ has multiplicity $4k$. Let denote by $\bar U_i$ the $\omega^K$-invariant $4k_i$-dimensional subspaces associated to $\theta_i \neq \pi/2$ and denote by $d$ the sum of their dimensions.
Any such subspace  is $I$-complex. In fact  by the uniqueness of such invariant $\omega^K$-subspaces and from   Proposition (\ref{bases of invariat 2-planes of omega are I-ortogonal}), one has the following decomposition into 4-dimensional $I$-complex subspaces $(U_{ij},I,\theta_i)$
\[\bar U_i= \bigoplus_{j=1}^{k_i} (U_{ij},I,\theta_i)= \bigoplus_{j=1}^{k_i}   L(X_{ij}, Z_{ij}=\frac{K^{-1}Pr_U^{KU}X_{ij}}{\cos \theta^K}, IX_{ij},IZ_{ij}), \qquad X_{ij} \in \bar U \cap (\bigoplus_{p=1}^{j-1} U_{ip})^\perp.\]
The union of the bases above is an $\omega^K$-standard basis of the $I$-complex  subspace $\bar U_i$.

Denoting by $W= \bigoplus \bar U_i$, from Claim  (\ref{properties of some quternionic,complex and real subspaces}), we have that $W$ is a $d$-dimensional $I$-complex subspace, it admits a decomposition into 4-dimensional $I$-complex subspaces  and, w.r.t. the orthonormal bases given above, the form $\omega^K|_W$ assumes standard form.
The 4-dimensional complex addends of $W$ are Hermitian orthogonal as can be easily seen. In fact, if $W_1,W_2$ are a pair of such addends, one has that  $W_1 \perp  W_2=IW_2$ and $W_1 \perp JW_2=KW_2$  i.e. $ W_1^\mathbb H \perp  W_2^\mathbb H$. In other words, all different 4-dimensional $I$-complex addends $U_{ij}$ of the pure $I$-complex subspace $U$ belong to 8-dimensional quaternionic subspaces orthogonal in pairs. 

Finally if $2m-d>0$, the  $(2m-d)$-dimension subspace $W'=W^\perp \cap  U$ is $I$-complex being the orthogonal complement to a complex subspace in a complex space (see claim (\ref{properties of some quternionic,complex and real subspaces})). In particular it is totally $I$-complex since the angle of isoclinicity between the pair $(W',KW')$ is $\pi/2$. It is easy to see that  $W'$ is Hermitian orthogonal to $W$. Furthermore the $\omega^I$ standard form restricted to $W'$  determine a decomposition of $W'$ into Hermitian orthogonal 2-dimensional totally $I$-complex addends.
Summing them in pairs the conclusion follows.

\end{proof}

Although the decomposition on the previous Theorem is  unique only if
all $\cos \theta_i>0$  have multiplicity 4 and eventually present  $\cos \theta_i=0$ have multiplicity 2, we can state the following corollary  whose proof is straightforward.
\begin{coro} \label{Quaternionic Kehler multiangle}
To a pure  $I$-complex subspace  
 $(U^{4m},I)$ 
 we can  canonically associate the vector $\boldsymbol{\theta^{I^\perp}}=(\theta_1^{I^\perp}, \ldots, \theta_{m}^{I^\perp})$  where $\theta_i^{I^\perp}$ (ordered in increasing order) are the  $I^\perp$-K\"{a}hler angles of the Hermitian orthogonal 4-dimensional $I$-complex subspaces of Theorem (\ref{canonical decomposition of a complex subspace}).   If $\dim U=4m+2$, the angle $\theta_{m+1}^{I^\perp}= \pi/2$ is the $K$-K\"{a}hler angle of  an Hermitian orthogonal  totally complex 2-plane.
\end{coro}
The increasing order   of the  $I^\perp$-K\"{a}hler angles in $\boldsymbol{\theta}$ determines a corresponding order of the relative  Hermitian orthogonal 4-dimensional $I$-complex subspaces.




 \begin{defi}
Let $(U,I)$ be a $2m$-dimensional $I$-complex subspace. We call the vector  $\boldsymbol{\theta}^{I^\perp}=(\theta_1^{I^\perp}, \ldots, \theta_{[m/2]}^{I^\perp})$ ($\boldsymbol{\theta}=(\theta_1^{I^\perp}, \ldots, \theta_{{m/2}^{I^\perp}}, \pi/2)$ if $m$ is odd)  with $\theta_i^{I^\perp}$ ordered in increasing order, the \textbf{$I^\perp$-K\"{a}hler multipleangle} of the $I$-complex $2m$-dimensional subspace $(U,I)$ that we will denote by $(U^{2m},I,\boldsymbol{\theta^{I^\perp}})$.
 \end{defi}

For any leading vector $X_1=Y_1=Z_1$, we associate to any 4-dimensional complex subspace the chains $\{X_i\},\{Y_i\},\{Z_i\}$ of the  Definition (\ref{associated chains}) and given in (\ref{the chains of a 4dimensionalcomplex subspace}).  We recall that in case $U$ is totally complex 
one has that  $\{X_i \}=\{\tilde X_i \}=\{Y_i \}= \{\tilde Y_i \}=\{Z_i \}= \{\tilde Z_i \}$.

\begin{defi} \label{canonical bases of a 2m complex subspace}
The unions of the chains $\{X_i\}$ (resp. $\{Y_i\},\{Z_i\}$ of the 4-dimensional $I$-complex addends $(U_{ij},I,\theta_i^{I^\perp})$ form the triple of    the  canonical bases of $(U^{2m},I,\boldsymbol{\theta^{I^\perp}})$.
\end{defi}
\begin{prop} \label{4 by 4 diagonal form of the canonical matrices of a 2m complex subspace}
Let $(U^{2m},I,\boldsymbol{\theta}^{I^\perp})$ a $2m$-dimensional $I$-complex subspace.  If $m$ is even, the canonical matrix $C_{IJ}$ (resp. $C_{IK}$), w.r.t. the canonical bases,
is given by a diagonal block matrices with $4 \times 4$-blocks given by the  first (resp. the second) of the (\ref{canonical matrices of a 4 dimensional complex subspace}) (plus an order  2 identity block if $m$ is odd).
\end{prop}


\begin{proof}
From Theorem (\ref{canonical decomposition of a complex subspace}) we have that $U$ admits a standard decomposition into 4-dimensional $I$-complex subspaces. The addends are Hermitian orthogonal then every orthogonal change of basis preserving such decomposition is represented by a diagonal block matrix with $4 \times 4$ blocks.  The conclusion follows from Proposition (\ref{the vectors $X_2,Y_2,Z2$ for an orthonormal nasis of thge orthogonal complement of $X_1$}).
\end{proof}

Let denote by $Gr_{(I,\boldsymbol{\theta}^{I^\perp})}^\R(2m,4n)$ the set of $2m$-dimensional  pure $I$-complex subspaces in $(V^{4n}, <,>, \mathcal{Q})$ with $I^\perp$-K\"{a}hler multipleangle  $\boldsymbol{\theta}^{I^\perp}$.

\begin{teor} \label{orbit of a complex 2m-plane under Sp(n)}

The group $Sp(n)$ acts transitively on $Gr_{(I,\boldsymbol{\theta}^{I^\perp})}^\R(2m,4n)$ i.e.
the pair $(I, \boldsymbol{\theta}^{I^\perp})$ composed by the complex structure $I \in \mathcal{Q}$ and the $I^\perp$-K\"{a}hler multipleangle  $\boldsymbol{\theta}^{I^\perp}$ of the $I$-complex subspace $U$ determines completely its $Sp(n)$-orbit in the Grassmannian $Gr^\R(2m,4n)$.
\end{teor}



\begin{proof}
The  proof follows directly from Proposition  (\ref{canonical decomposition of a complex subspace}). The Hermitian orthogonality of all the addends of the decomposition of a $2m$-dimensional pure $I$-complex subspace  $(U^{2m},I, \boldsymbol{\theta}^{I^\perp})$ there stated,   allows us to deal separately with each addend since the group $Sp(n)$ preserves such orthogonality.  Therefore the canonical matrices w.r.t. the canonical bases given in the Definition (\ref{canonical bases of a 2m complex subspace}) have  the  unique  form  stated  in the Proposition (\ref{4 by 4 diagonal form of the canonical matrices of a 2m complex subspace}), then  from the Theorem (\ref{main_theorem 1_ with respect canonical bases}),  the pair $(I, \boldsymbol{\theta}^{I^\perp})$ determines the $Sp(n)$-orbits of $U$.

In case $\dim U$ is not multiple of 4, the last addend   of the Hermitian orthogonal decomposition stated in Proposition (\ref{canonical decomposition of a complex subspace}) is a  totally $I$-complex 2-plane and the conclusion follows from Proposition (\ref{Sp(n) orbita di un 2 piano complesso}).
\end{proof}

 If   $U$ is $A$-complex  with $A \in S(\mathcal{Q})$ the triple $(\xi,\chi,\eta) \neq (0,0,0)$  w.r.t. an the admissible basis $(I,J,K)$  as stated in the Proposition (\ref{xi=chi=eta=0 of a complex 4-plane  only w.r.t. an adapted basis}) unless $A= \pm I$. Then the canonical matrices never have the form stated  in the Proposition (\ref{4 by 4 diagonal form of the canonical matrices of a 2m complex subspace}).

\subsection{$Sp(n)$-orbit of a $\Sigma$-complex subspace}
In Proposition (\ref{unicity of the decomposition of a subspace with trivial real addend}) we stated that a $\Sigma$-complex subspace  $U$ admits a unique decomposition into Hermitian orthogonal sum of maximal pure complex subspaces by different complex structure. Although in Theorem  (\ref{canonical decomposition of a complex subspace}) we stated that the decomposition of each $I_i$-complex $2m$-dimensional addend into  4-dimensional Hermitian orthogonal complex addends (if $m \geq 2$) plus eventually an Hermitian orthogonal totally complex plane (if $m$ is odd)  is in general not unique, we have that the $I_i^\perp$- K\"{a}hler multipleangle is canonically defined.
In this case, from Corollary (\ref{orthogonal complex subspaces by different structures}),  we can determine the $Sp(n)$-orbit of $U$ by determining separately the orbit of each complex addend.

Let $U=\bigoplus_{i=1}^s (U_i^{2m_i}, I_i, \boldsymbol\theta^{I_i^\perp})$
 where $\boldsymbol \theta^{I_i^\perp}=(\theta_1^{I_i^\perp}, \theta_2^{I_i^\perp}, \ldots, \theta_{[m_i/2]}^{I_i^\perp})$  is the $I_i^\perp$-K\"{a}hler multipleangle of the $I_i$-complex subspace $(U_i,I_i,\boldsymbol\theta^{I_i^\perp})$   whose elements are the  $\theta^{I_i^\perp}$-K\"{a}hler angles of the 4-dimensional Hermitian orthogonal addends 
(plus eventually an Hermitian orthogonal  totally $I_i$-complex  plane if $m_i$ is odd). Denote by $\mathcal{I}=:(I_1,I_2, \ldots, I_s)$ the vector of the complex structures of the different complex addends $(U_i,I_i)$ ordered as stated in section (\ref{Decomposition})  and by  $\Theta:=(\boldsymbol\theta^{I_1^\perp}, \ldots, \boldsymbol\theta^{I_s^\perp})$ the vector whose elements are the  respective $I_i^\perp$-K\"{a}hler multipleangle of each $U_i$.  We can then state the
 \begin{teor} \label{orbit af a generic subspace with trivial real part under Sp(n)}
The $Sp(n)$-orbit of the $\Sigma$-complex subspace $U$ is completely determined by the pair $(\mathcal{I}, \Theta)$.
 \end{teor}

\begin{proof}
Again the Hermitian orthogonality of the complex 4-dimensional subspaces allows us to consider the orbit of each of them separately.  The canonical matrices of each $(U_i,I_i)$, w.r.t. an adapted basis  and    w.r.t. a canonical basis have the form stated in the Proposition (\ref{4 by 4 diagonal form of the canonical matrices of a 2m complex subspace}).
For any admissible basis  and from Corollary (\ref{same invariants w.r.t. one admissible basis implies same invariants w.r.t. any admissible basis}) the canonical matrices of the $\Sigma$-complex subspace $U$ w.r.t. the union of the canonical basis of each $I_i$-complex subspace have then a unique form.
The conclusion follows from the Theorem (\ref{main_theorem 1_ with respect canonical bases}).
\end{proof}

 In particular this is true for any $I_1$-complex subspace in which case  $\mathcal{I}=I_1$ and $\Theta=\boldsymbol \theta^{I_1^\perp}$.


\end{document}